\documentclass[11pt]{amsart}
\usepackage[margin=1in]{geometry}

\usepackage{amssymb}
\usepackage{amsthm}
\usepackage{amsmath}
\usepackage{mathrsfs}
\usepackage{amsbsy}
\usepackage{bm}
\usepackage{hyperref}
\usepackage{tikz}
\usepackage{array}
\usepackage{enumerate}
\usepackage{xcolor}
\usepackage{bbm}
\usepackage{comment}
\usepackage{mathtools}

\definecolor{lavender}{rgb}{0.4,0,1}
\definecolor{Teal}{rgb}{0,0.784,0.784}

\usepackage[noabbrev,capitalise,nameinlink]{cleveref}

\hypersetup{colorlinks=true, citecolor=lavender, linkcolor=lavender,urlcolor=lavender}

\usepackage{caption}
\usepackage[noabbrev,capitalise,nameinlink]{cleveref}
\crefname{conjecture}{Conjecture}{Conjectures}

\newtheorem{theorem}{Theorem}[section]
\newtheorem{proposition}[theorem]{Proposition}
\newtheorem{corollary}[theorem]{Corollary}
\newtheorem{conjecture}[theorem]{Conjecture}

\newtheorem{lemma}[theorem]{Lemma}

\theoremstyle{definition}
\newtheorem{definition}[theorem]{Definition}
\newtheorem{remark}[theorem]{Remark}
\newtheorem{example}[theorem]{Example} 

\newcommand{\BB}{\mathbb{B}}
\newcommand{\Traj}{\mathrm{Traj}}
\newcommand{\ZZ}{\mathbb{Z}}
\newcommand{\WW}{\mathcal{W}}
\newcommand{\calH}{\mathcal{H}}
\newcommand{\HH}{\mathrm{H}}
\newcommand{\SSS}{\mathfrak{S}}
\newcommand{\affS}{\widetilde{\mathfrak{S}}}
\newcommand{\TT}{\mathbb{T}}
\newcommand{\qq}{\mathbf{q}}
\newcommand{\Cycle}{\mathsf{Cycle}}

\usepackage{etoolbox}

\newcommand{\includeSymbol}[1]{\ensuremath{%
	\mathchoice
		{\raisebox{-.7mm}{\includegraphics[height=2.2ex]{#1}}}	
		{\raisebox{-.7mm}{\includegraphics[height=2.2ex]{#1}}}
		{\raisebox{-.6mm}{\includegraphics[height=1.6ex]{#1}}}
		{\raisebox{-.5mm}{\includegraphics[height=1ex]{#1}}}
}}

\robustify{\includeSymbol}
\newcommand{\Reflect}{\includeSymbol{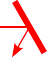}}

\robustify{\includeSymbol}
\newcommand{\Refract}{\includeSymbol{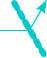}}

\newcommand{\Eflect}{{\color{red}E_{\Reflect}}}
\newcommand{\Efract}{{\color{Teal}E_{\Refract}}}

\newcommand{\dfn}[1]{\textcolor{blue}{\emph{#1}}}

\DeclareMathOperator{\compl}{compl}

\begin{document}

\title[]{Homology in Combinatorial Refraction Billiards}
\subjclass[2010]{}

\author[]{Colin Defant}
\address[]{Department of Mathematics, Harvard University, Cambridge, MA 02138, USA}
\email{colindefant@gmail.com}

\author[]{Derek Liu}
\address[]{Department of Mathematics, Massachusetts Institute of Technology, Cambridge, MA 02139, USA}
\email{derekl@mit.edu}

\begin{abstract}
Given a graph $G$ with vertex set $\{1,\ldots,n\}$, we can project the graphical arrangement of $G$ to an $(n-1)$-dimensional torus to obtain a toric hyperplane arrangement. Adams, Defant, and Striker constructed a toric combinatorial refraction billiard system in which beams of light travel in the torus, refracting (with refraction coefficient $-1$) whenever they hit one of the toric hyperplanes in this toric arrangement. Each billiard trajectory in this system is periodic. We adopt a topological perspective and view the billiard trajectories as closed loops in the torus. We say $G$ is \emph{ensnaring} if all of the billiard trajectories are contractible, and we say $G$ is \emph{expelling} if none of the billiard trajectories is contractible. Our first main result states that a graph is expelling if and only if it is bipartite. We then provide several necessary conditions and several sufficient conditions for a graph to be ensnaring. For example, we show that the complement of an ensnaring graph cannot have a clique as a connected component. We also discuss ways to construct ensnaring graphs from other ensnaring graphs. For example, gluing two ensnaring graphs at a single vertex always yields another ensnaring graph.
\end{abstract} 

\maketitle

\section{Introduction}\label{sec:intro} 

\subsection{Combinatorial Refraction Billiards}
Over the past several years, there has been a flurry of attention devoted to \dfn{refraction billiards} (also called \dfn{tiling billiards}), a variant of mathematical billiards in which a beam of light refracts (i.e., bends) instead of reflecting when it hits a hyperplane \cite{Baird,Barutello,Davis1,Davis2,DeBlasi2,DeBlasi1,Jay,Paris}. In this setting, the refraction coefficient (see \cite{Baird,Davis1,Davis2}) is taken to be $-1$; this means that the new direction that the beam of light faces after it refracts through a hyperplane is the opposite of the direction it would have faced if it had reflected (see \cref{fig:refract}). 

\begin{figure}[ht]
  \begin{center}
  \includegraphics[height=1.3cm]{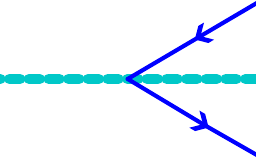}
  \end{center}
\caption{A beam of light refracts through a horizontal line. }\label{fig:refract}
\end{figure}

\dfn{Combinatorial billiards} is a new topic that merges ideas from dynamical algebraic combinatorics and mathematical billiards; it concerns billiard systems that are rigid and discretized in a way that allows them to be modeled combinatorially or algebraically \cite{ADS,Barkley,BDHKL,DefantStoned,DefantJiradilok,DJM,Zhu}. Adams, Defant, and Striker \cite{ADS} recently introduced combinatorial refraction billiard systems, which live in high-dimensional Euclidean spaces or tori. This is in contrast to non-combinatorial refraction billiard systems from the literature, which almost exclusively live in the $2$-dimensional Euclidean plane. Indeed, the rigid constraints placed on combinatorial billiard systems allow one to prove very precise results in high-dimensional settings, whereas classical billiards is often concerned with more analytical problems in low-dimensional settings.  

Let \[V=\{(\gamma_1,\ldots,\gamma_n)\in\mathbb R^n:\gamma_1+\cdots+\gamma_n=0\}.\] Let $e_i$ denote the $i$-th standard basis vector in $\mathbb R^n$. Let $[n]=\{1,\ldots,n\}$. For $r\in[n]$, consider the vector $\delta^{(r)}\coloneq -ne_r+\sum_{i\in[n]}e_i$ whose $r$-th coordinate is $-(n-1)$ and whose other coordinates are all $1$. Let ${\Delta=\{\pm\delta^{(1)},\ldots,\pm\delta^{(n)}\}}$. For $1\leq i<j\leq n$ and $k\in\ZZ$, let \[\HH_{i,j}^k=\{(\gamma_1,\ldots,\gamma_n)\in V:\gamma_i-\gamma_j=k\}.\] Let \begin{equation}\label{eq:calH}
\calH_n=\{\HH_{i,j}^k:1\leq i<j\leq n,\,k\in\ZZ\}.
\end{equation} 
As discussed in \cref{sec:prelim}, $\calH_n$ is the \emph{Coxeter arrangement} of the affine symmetric group $\affS_n$. 

Fix a subset $\WW\subseteq\calH_n$. The elements of $\WW$ will serve as our refractive hyperplanes, so we call them \dfn{metalenses}\footnote{In the physical world,
materials typically have positive refraction coefficients. However, physicists and material scientists have discovered certain \emph{metamaterials}, which have negative refraction coefficients \cite{meta}. This inspired Adams, Defant, and Striker to introduce the term \emph{metalens} \cite{ADS}.}. Hyperplanes in $\calH_n\setminus\WW$ are called \dfn{windows}. Starting at a point in $V$ that is not in any of the hyperplanes in $\calH_n$, we can shine a beam of light in the direction of one of the vectors in $\Delta$. Whenever the beam of light hits a metalens, it refracts; otherwise, it always moves in a straight line. Let $\widetilde{\Traj}_\WW$ be the set of billiard trajectories defined in this manner. (Formally, elements of $\widetilde{\Traj}_\WW$ are piecewise-linear curves in $V$.) See the top images in \cref{fig:path3,fig:K3} for examples. One can show that the beam of light always travels in the direction of a vector in $\Delta$ (except, of course, at the exact moments when it is refracting) and that the trajectory of the beam of light is bounded if and only if it is periodic. 

\subsection{Toric Billiards}

The \dfn{coroot lattice} of $\affS_n$ is the lattice $Q^\vee=V\cap\ZZ^n$. We will work with the $(n-1)$-dimensional torus $\TT\coloneq V/Q^\vee$. Let $\qq\colon V\to\TT$ be the natural quotient map. For $\HH\in\calH_n$, the image $\qq(\HH)$ of $\HH$ under the quotient map is called a \dfn{toric hyperplane}. 

The map $\qq$ projects two hyperplanes $\HH_{i,j}^k$ and $\HH_{i',j'}^{k'}$ (with $i<j$ and $i'<j'$) to the same toric hyperplane if and only if $i=i'$ and $j=j'$. Thus, the map $(i,j)\mapsto\qq(\HH_{i,j}^0)$ is a bijection from $\{(i,j):1\leq i<j\leq n\}$ to the set of toric hyperplanes in our toric hyperplane arrangement. This allows us to encode a collection of toric hyperplanes as a (simple) graph $G=([n],E)$ with vertex set $[n]$; a toric hyperplane $\qq(\HH_{i,j}^0)$ is in the collection if and only if $i$ and $j$ are adjacent in $G$. 

We are primarily interested in defining a combinatorial refraction billiard system by projecting the billiard trajectories described above from the Euclidean space $V$ into the torus $\TT$. In order for this to make sense, we need to ensure that two hyperplanes $\HH,\HH'\in\calH_n$ satisfying $\qq(\HH)=\qq(\HH')$ are either both windows or both metalenses. Thus, we will assume from now on that the collection $\WW\subseteq\calH_n$ is of the form 
\[\WW_G=\{\HH_{i,j}^k:\{i,j\}\in E,\,k\in\ZZ\},\] where $G=([n],E)$ is a graph. Under this assumption, it makes sense to define a \dfn{toric metalens} (respectively, \dfn{toric window}) to be a toric hyperplane of the form $\qq(\HH)$ such that $\HH\in\WW_G$ (respectively, $\HH\not\in\WW_G$). Let $\Traj_G$ be the set of toric billiard trajectories obtained by applying $\qq$ to billiard trajectories in $\widetilde{\Traj}_{\WW_G}$.  

\subsection{Homology} 
The main novel contribution of the present article is a topological perspective on combinatorial billiards. Because the beams of light in our Euclidean billiard system always travel in one of the directions in $\Delta$, one can show that all of the billiard trajectories in our toric billiard system are periodic. This means that the trajectories in $\Traj_G$ are closed loops in $\TT$; we are interested in understanding the homology classes of these loops. In particular, we are interested in whether or not such loops are contractible. 

\begin{definition}\label{def:ensnaring}
Let $n\geq 2$, and let $G=([n],E)$ be a simple graph. We say $G$ is \dfn{ensnaring} if all of the trajectories in $\Traj_G$ are contractible. We say $G$ is \dfn{expelling} if none of the trajectories in $\Traj_G$ is contractible. 
\end{definition} 

\begin{remark}
Each trajectory in $\Traj_G$ can be lifted to paths in $\widetilde{\Traj}_{\WW_G}$. A trajectory in $\Traj_G$ is contractible if and only if all of its lifts are bounded (equivalently, periodic). Hence, a graph is ensnaring (respectively, expelling) if and only if all (respectively, none) of the trajectories in $\widetilde{\Traj}_{\WW_G}$ are bounded. 
\end{remark}

\begin{remark}
The collection of homology classes of loops in $\Traj_G$ only depends on the isomorphism class of $G$, so we will often use the terms \emph{ensnaring} and \emph{expelling} to refer to isomorphism classes of graphs (without explicitly identifying the vertices with the elements of $[n]$). 
\end{remark} 

\begin{remark}\label{rem:1vertex} 
Although our geometric perspective using billiards does not make sense when $n=1$, we will make the convention that a graph with only $1$ vertex is both ensnaring and expelling. 
\end{remark}

\begin{remark}\label{rem:2vertices}
The simplest case to consider is when $n=2$, although this case is so simple as to be potentially confusing. Here, the space $V$ is $1$-dimensional, and the hyperplanes in $\mathcal H_n$ are points. When a beam of light refracts through a hyperplane in $\mathcal H_n$, it simply passes through as it would if it had not refracted. In other words, metalenses behave exactly the same as windows when $n=2$. It follows that a graph with $2$ vertices is expelling (and not ensnaring). 
\end{remark} 

\begin{remark}
We will see that there exist graphs that are neither ensnaring nor expelling. 
\end{remark}

\subsection{Main Results and Outline} 

In \cref{sec:prelim}, we provide necessary background on the affine symmetric group and its standard geometric representation. We also discuss a purely combinatorial model for our toric combinatorial refraction billiard systems and explain how to determine from this model whether or not a billiard trajectory is contractible. 

\cref{sec:expelling} is devoted to proving our first main result, which states that a graph is expelling if and only if it is bipartite (\cref{thm:expelling}). 

The bulk of our work is devoted to understanding ensnaring graphs, which appear to be far more unwieldy than expelling graphs. While we will not completely settle this problem, we will find several necessary conditions and sufficient conditions for a graph to be ensnaring. 

\cref{sec:completecycle} concerns two natural families of graphs: complete graphs and cycles. We prove that a complete graph with at least $3$ vertices is ensnaring (\cref{thm:complete}). We also show that a cycle graph with $n\geq 3$ vertices is ensnaring if and only if $n$ is odd (\cref{thm:cycles}). 

In \cref{sec:wedgeunion}, we discuss two operations that can be used to build larger graphs from smaller ones. First, we consider a \dfn{wedge} of two graphs $G_1$ and $G_2$, which is a graph obtained by gluing a single vertex of $G_1$ to a vertex of $G_2$. We prove that a wedge of two ensnaring graphs is ensnaring (\cref{thm:wedge}). We also conjecture that a wedge of two graphs that are not ensnaring is not ensnaring (\cref{conj:wedge_nonensnaring}). We then consider the disjoint union of two graphs, which, from our perspective in this article, is not as easy to understand as one might think. We use the combinatorial model for toric combinatorial refraction billiard systems discussed in \cref{sec:prelim} to define yet another family of graphs that we call \emph{revolutionary}, and we prove that the disjoint union of two graphs is ensnaring if and only if both of the graphs are ensnaring but not revolutionary (\cref{thm:union}). We also prove that if $n\geq 5$ is odd, then the cycle graph with $n$ vertices is revolutionary (\cref{thm:revcycles}). 

In \cref{sec:wedgecomplete}, we prove that a graph formed by wedging trees to the vertices of a complete graph with at least $3$ vertices is ensnaring. Since trees are expelling (and hence, not ensnaring), this shows, in particular, that a wedge of an ensnaring graph and a non-ensnaring graph can be ensnaring. 

\cref{sec:nonensnaring} discusses a local obstruction that prevents a graph from being ensnaring. More precisely, 
\cref{thm:local_not_ensnaring} provides a configuration of $4$ vertices satisfying some conditions that, if present in a graph, ensures the graph cannot be ensnaring.  

The \dfn{complement} of a graph $G$, denoted $\overline G$, is the graph with the same vertex set as $G$ defined so that two vertices are adjacent in $\overline G$ if and only if they are not adjacent in $G$. For integers $m\leq n$, we define the \dfn{$n$-vertex complement} of an $m$-vertex graph $H$ to be the graph whose complement is the disjoint union of $H$ with $n-m$ isolated vertices. In \cref{sec:connected_complement}, we prove that a graph with $n$ vertices is ensnaring if and only if the $n$-vertex complement of each of the connected components of $\overline G$ is ensnaring (\cref{thm:complcomp}). We show that for $2\leq m\leq n$, the $n$-vertex complement of a complete graph on $m$ vertices is not ensnaring (\cref{thm:complcomplete}). We also prove that for integers $\ell,r,n$ with $\ell+r<n$, the $n$-vertex complement of the complete bipartite graph $K_{\ell,r}$ is ensnaring if and only if $n-\ell-r$ is odd and $(\ell,r)\neq(1,1)$. 

Finally, in \cref{sec:future}, we collate several conjectures and suggestions for future work. This includes a generalization of our setup involving toric combinatorial billiard systems in which each toric hyperplane is a window, a metalens, or a \emph{mirror}. A mirror is a third type of hyperplane that causes a beam of light to reflect. 

\section{Preliminaries}\label{sec:prelim}

Let $\SSS_n$ denote the \dfn{symmetric group} consisting of permutations of the set $[n]$. The \dfn{affine symmetric group} is the group of bijections $u\colon\ZZ\to\ZZ$ such that 
\begin{itemize}
\item $u(i+n)=u(i)+n$ for all $i\in\ZZ$;
\item $u(1)+\cdots+u(n)=\binom{n+1}{2}$. 
\end{itemize}
The \dfn{simple reflection} $s_i$ is the element of $\affS_n$ that swaps $i+kn$ and $i+1+kn$ for all $k\in\ZZ$. Since $s_i=s_{i+n}$ for every $i$, we can view the indices of the simple reflections as elements of $\ZZ/n\ZZ$. Then $(\affS_n,\{s_i:i\in\ZZ/n\ZZ\})$ is a Coxeter system. 

The hyperplane arrangement $\calH_n$ in \eqref{eq:calH} is the \dfn{Coxeter arrangement} of $\affS_n$. There is a faithful right action of $\affS_n$ on $V$ in which $s_0$ acts as the reflection through the hyperplane $\HH_{1,n}^1$ and each $s_i$ for $i\in[n-1]$ acts as the reflection through the hyperplane $\HH_{i,i+1}^0$. The closures of the connected components of $V\setminus\bigcup_{\HH\in\calH_n}\HH$ are called \dfn{alcoves}. The \dfn{fundamental alcove} is 
\[\BB=\{(\gamma_1,\ldots,\gamma_n)\in V:\gamma_1\geq\gamma_2\geq\cdots\geq\gamma_n\geq\gamma_1-1\}.\]
The map $w\mapsto\BB w$ is a bijection from $\affS_n$ to the set of alcoves. Two alcoves $\BB w$ and $\BB w'$ are adjacent (i.e., share a common facet) if and only if $w'=s_iw$ for some $i\in\ZZ/n\ZZ$. For $u\in\affS_n$ and $i\in\ZZ/n\ZZ$, we write $\HH^{(u,s_i)}$ for the unique hyperplane separating $\BB u$ from $\BB s_iu$. 

Let $G=([n],E)$ be a simple graph, and let $\WW_G=\{\HH_{i,j}^k\in\calH_n:\{i,j\}\in E\}$ be the corresponding set of metalenses. Let $\widetilde\Xi_n=\affS_n\times\ZZ/n\ZZ\times\{\pm 1\}$. Following \cite{ADS}, we define $\widetilde\Theta_G\colon\widetilde\Xi_n\to\widetilde\Xi_n$ by 
\[\widetilde\Theta_G(w,i,\epsilon)=\begin{cases} (s_i w,i+\epsilon,\epsilon) & \mbox{if }\HH^{(w,s_i)}\not\in \WW_G; \\   (s_iw,i-\epsilon,-\epsilon) & \mbox{if }\HH^{(w,s_i)}\in \WW_G.  \end{cases}\] 
Starting with a triple $(u_0,i_0,\epsilon_0)$, let $\widetilde\Theta_G^k(u_0,i_0,\epsilon_0)=(u_k,i_k,\epsilon_k)$. Shine a beam of light from a point in the interior of $\BB u_0$ in the direction of the vector $\epsilon_{0}\delta^{(i_0+\frac{1}{2}(1-\epsilon_0))}u_0\in\Delta$. Whenever the beam of light hits a window, let it pass straight through; whenever it hits a metalens, let it refract. As discussed in \cite{ADS}, after the beam of light has passed through a hyperplane in $\calH_n$ for the $M$-th time, it will be inside the alcove $\BB u_M$ facing in the direction of $\epsilon_M\delta^{(i_M+\frac{1}{2}(1-\epsilon_M))}u_M$. Thus, we can think of the sequence $(\BB u_M)_{M\geq 0}$ as a discretization of the beam of light. The trajectory of the beam of light is periodic if and only if this sequence is periodic. 

For each coroot vector $\xi\in Q^\vee=V\cap\ZZ^n$, there is an element of $\affS_n$ that acts via translation by $\xi$. This allows us to view $Q^\vee$ as an abelian normal subgroup of $\affS_n$. The affine symmetric group decomposes as the semidirect product $\affS_n\cong\SSS_n\ltimes Q^\vee$. There is a natural quotient map $\affS_n\to\SSS_n$, which we denote by $w\mapsto\overline w$. Let $\TT=V/Q^\vee$. For every $w\in\affS_n$, we have $\qq(\BB w)=\qq(\BB\overline w)$, where $\qq\colon V\to\mathbb T$ is the natural topological quotient map. For $i\in\ZZ/n\ZZ$, the permutation $\overline s_i$ is the transposition that swaps $i$ and $i+1$; in particular, $\overline s_0$ swaps $n$ and $1$. 

Let us now project the combinatorial refraction billiard system given by $\widetilde\Theta_G$ to the torus. Let $\Xi_n=\SSS_n\times\ZZ/n\ZZ\times\{\pm 1\}$. For $w\in\affS_n$ and $i\in\ZZ/n\ZZ$, the hyperplane $\HH^{(w,s_i)}$ is in $\WW_G$ if and only if $\{\overline w^{-1}(i),\overline w^{-1}(i+1)\}$ is an edge in $G$. Thus, we define $\Theta_G\colon\Xi_n\to\Xi_n$ by 
\[\Theta_G(v,i,\epsilon)=\begin{cases} (\overline s_i v,i+\epsilon,\epsilon) & \mbox{if }\{v^{-1}(i),v^{-1}(i+1)\}\not\in E; \\   (\overline s_iv,i-\epsilon,-\epsilon) & \mbox{if }\{v^{-1}(i),v^{-1}(i+1)\}\in E.  \end{cases}\] By construction, if $\widetilde\Theta_G(w,i,\epsilon)=(w',i',\epsilon')$, then $\Theta_G(\overline w,i,\epsilon)=(\overline{w'},i',\epsilon')$. Thus, we can view an orbit of $\Theta_G$ as a discretization of a billiard trajectory in the torus $\mathbb T$. 

In \cite{DefantToric}, Defant introduced a combinatorial dynamical system called \emph{toric promotion} as a cyclic analogue of Sch\"utzenberger's famous promotion operator. Adams, Defant, and Striker \cite{ADS} then interpreted toric promotion as a toric combinatorial billiard system, and they generalized it to a map called \emph{toric promotion with reflections and refractions} by adding refractions into the picture. The map $\Theta_G$ that we have defined is a special case of toric promotion with reflections and refractions in which there are no reflections, so we will simply call it \dfn{refractive toric promotion} (on $G$). 

Assume $n\geq 3$, and let $\Cycle_n$ be the cycle graph with vertex set $\ZZ/n\ZZ$, where each vertex $i$ is adjacent to $i-1$ and $i+1$; we embed $\Cycle_n$ in the plane so that $1,2,\ldots,n$ is the clockwise cyclic ordering of the vertices. For each vertex $a$ of $G$, we consider a formal symbol ${\bf a}$ called the \dfn{replica} of $a$.\footnote{We will always tacitly assume that a bolded version of a symbol representing a vertex is the replica of that vertex.} As in \cite{ADS}, we identify a triple $(v,i,\epsilon)\in\Xi_n$ with its \dfn{stone diagram}, which is constructed as follows: 
\begin{enumerate}
\item For each vertex $a$ of $G$, place the replica ${\bf a}$ on the vertex $v(a)$ of $\Cycle_n$. 
\item Place a stone on the vertex $i+\frac{1}{2}(1-\epsilon)$ of $\Cycle_n$. 
\item If $\epsilon=1$, orient the stone clockwise; if $\epsilon=-1$, orient the stone counterclockwise. 
\end{enumerate}
We will also frequently refer to the elements of $\Xi_n$ as \emph{stone diagrams}. In the stone diagram $(v,i,\epsilon)$, we say the stone \dfn{coexists with} the replica that sits on the vertex $i+\frac{1}{2}(1-\epsilon)$, and we say the stone \dfn{points toward} the replica that sits on the vertex $i+\frac{1}{2}(1+\epsilon)$. 

\begin{remark}
We only define stone diagrams when $n\geq 3$. This will not cause any problems because we already handled graphs with $1$ or $2$ vertices in \cref{rem:1vertex,rem:2vertices}. 
\end{remark}

The map $\Theta_G$ has a simple interpretation if we identify triples in $\Xi_n$ with their stone diagrams. Suppose the stone is sitting on vertex $j$ and has orientation $\epsilon$. When we apply $\Theta_G$, we swap the replicas sitting on vertices $j$ and $j+\epsilon$. If the vertices corresponding to these two replicas are adjacent in $G$, then we reverse the orientation of the stone and keep it in the same location; otherwise, we move the stone from vertex $j$ to vertex $j+\epsilon$ and do not change its orientation. 

\begin{example}\label{exam:path3}
Let $n=3$, and let $G$ be the graph \[\begin{array}{l}\includegraphics[height=0.672cm]{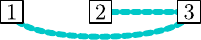}\end{array}.\] The top image of \cref{fig:path3} shows the Coxeter arrangement of $\affS_3$; metalenses are drawn in thick teal, while windows are in thin grey. Two alcoves are given the same color if and only if one is a translation of the other by a coroot vector. The top image also shows a forward orbit of $\widetilde\Theta_G$, where a triple $(u,i,\epsilon)\in\widetilde\Xi_3$ is represented by a blue arrow in the interior of $\BB u$ pointing in the direction of $\epsilon\delta^{(i+\frac{1}{2}(1-\epsilon))}u$. Note how the sequence of blue arrows discretizes a beam of light that refracts through metalenses and passes straight through windows. 

In the bottom image in \cref{fig:path3}, the torus $\mathbb T$ is drawn as a hexagon with opposite sides glued together. Each toric alcove is labeled by the corresponding permutation in $\SSS_3$. The bottom image shows the size-$18$ orbit of $\Theta_G$ obtained by projecting the orbit of $\widetilde\Theta_G$ from the top image to the torus. Each element of $\Xi_3$ is represented both as a blue arrow in the torus and as a stone diagram. The orbit of $\Theta_G$ discretizes a billiard trajectory of a beam of light that traverses a closed loop in the torus. The trajectory of the beam of light in the top image is a lift of the trajectory of the beam of light in the torus; since the former is not bounded, the closed loop traversed by the latter is not contractible.  
\end{example}

\begin{figure}[]
  \begin{center}
  \includegraphics[width=0.567\linewidth]{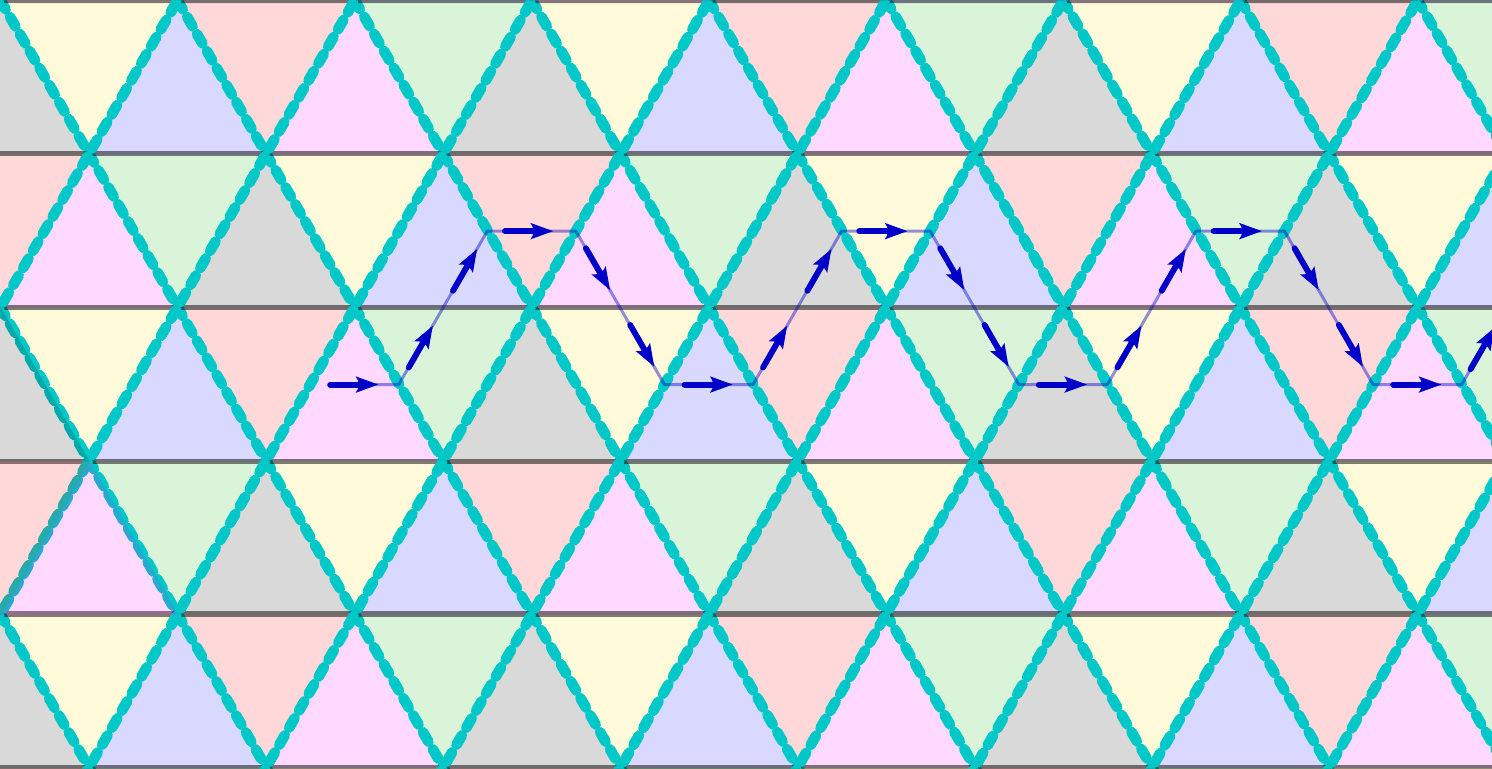} \\ \vspace{0.3cm}
  \includegraphics[width=\linewidth]{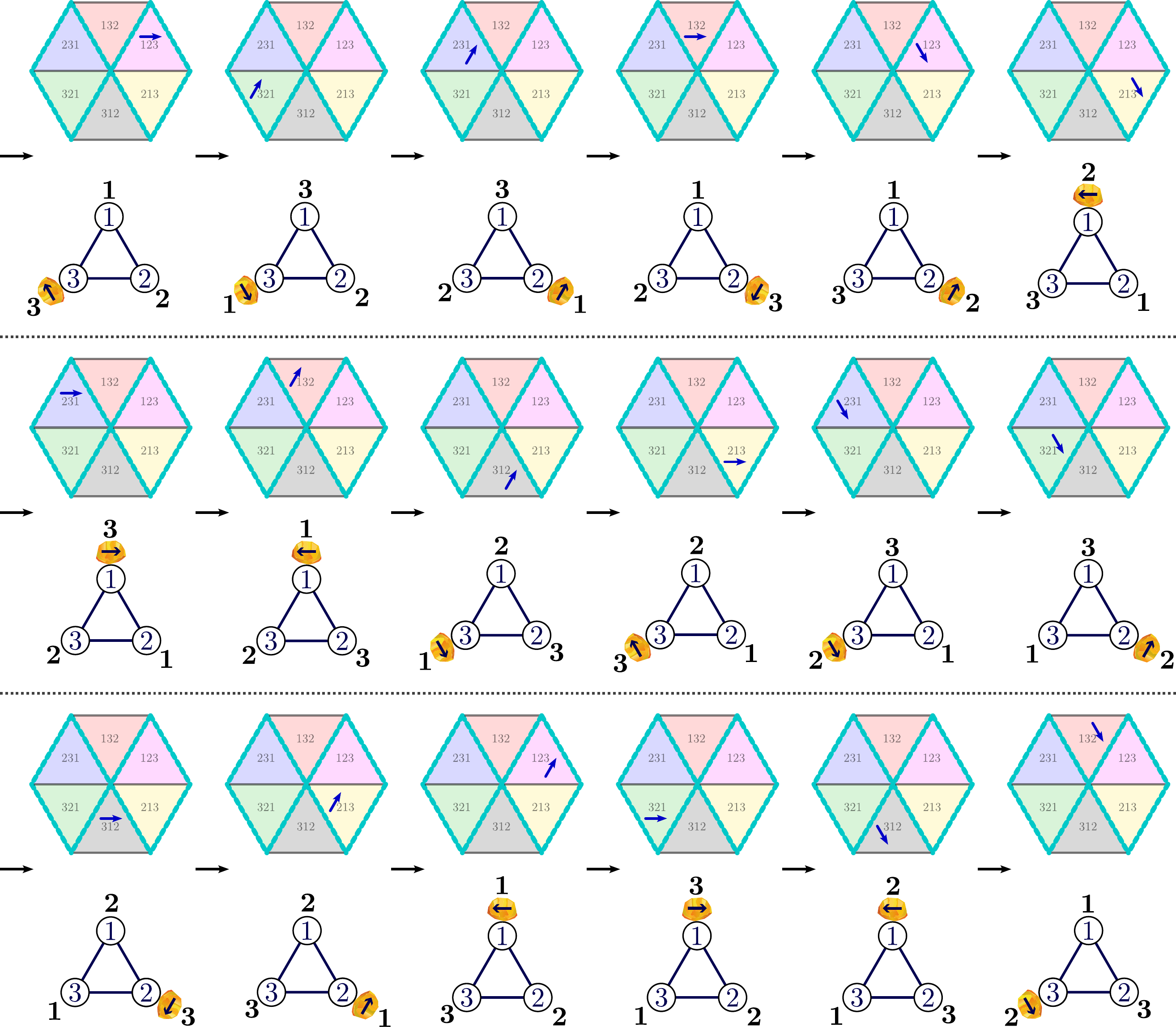}
  \end{center}
\caption{On the top is an unbounded combinatorial refraction billiard trajectory in $\affS_3$. On the bottom is the corresponding toric combinatorial refraction billiard trajectory, with each state represented both as an arrow in the torus and as a stone diagram.}\label{fig:path3}
\end{figure}

\begin{example}
Let $n=3$, and let $G$ be the graph \[\begin{array}{l}\includegraphics[height=0.672cm]{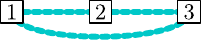}\end{array}.\] The top image of \cref{fig:K3} shows an orbit of $\widetilde\Theta_G$, where a triple $(u,i,\epsilon)\in\widetilde\Xi_3$ is represented by a blue arrow in the interior of $\BB u$ pointing in the direction of $\epsilon\delta^{(i+\frac{1}{2}(1-\epsilon))}u$. The sequence of blue arrows discretizes a beam of light that refracts every time it hits a hyperplane in $\calH_3$ since all such hyperplanes are metalenses. 

The bottom image in \cref{fig:K3} shows the size-$6$ orbit of $\Theta_{G}$ obtained by projecting the orbit of $\widetilde\Theta_{G}$ from the top image to the torus. The orbit of $\Theta_{G}$ discretizes a billiard trajectory of a beam of light that traverses a closed loop in the torus. The trajectory of the beam of light in the top image is a lift of the trajectory of the beam of light in the torus; since the former is bounded, the closed loop traversed by the latter is contractible.  
\end{example}

\begin{figure}[]
  \begin{center}
  \includegraphics[width=0.567\linewidth]{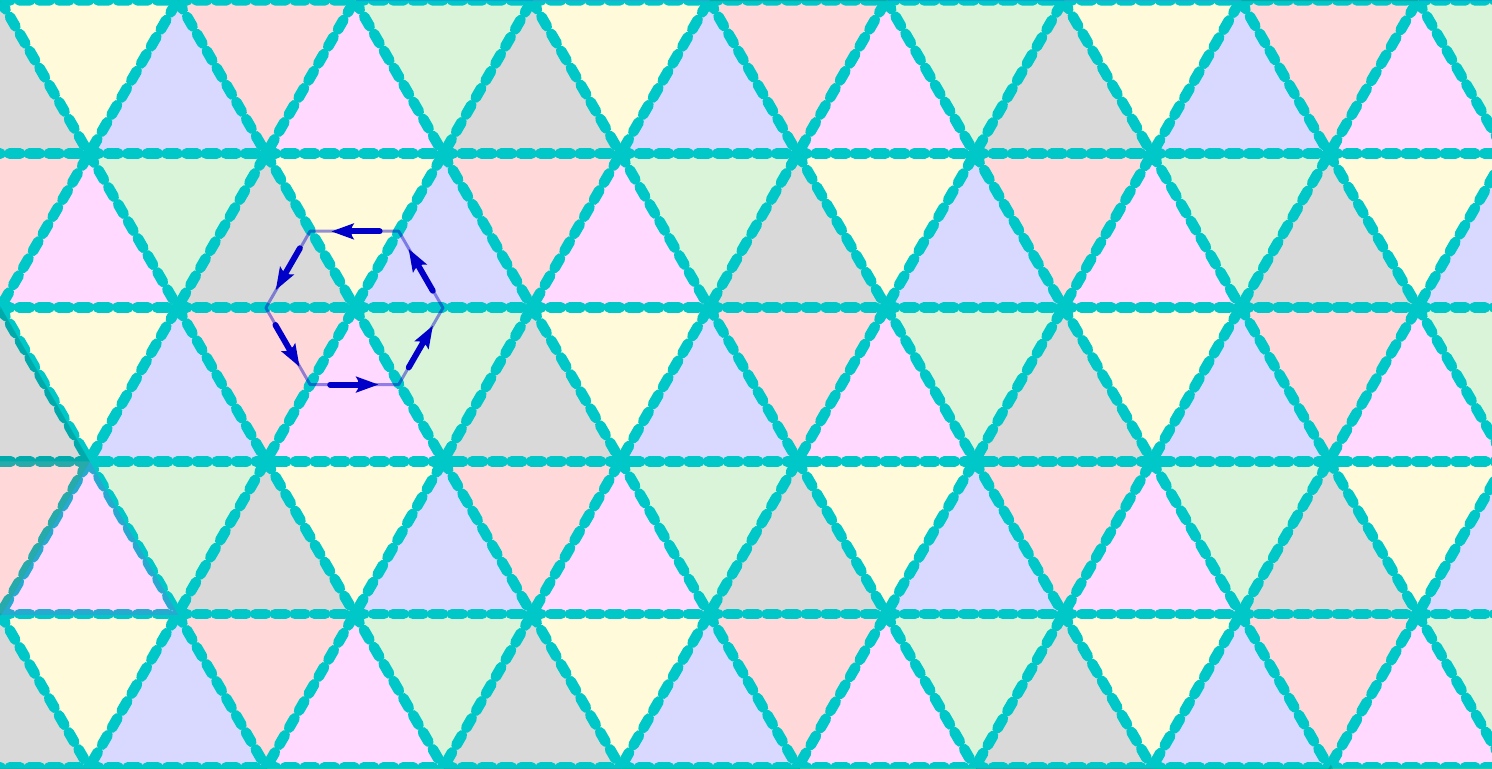} \\ \vspace{0.3cm}
  \includegraphics[width=\linewidth]{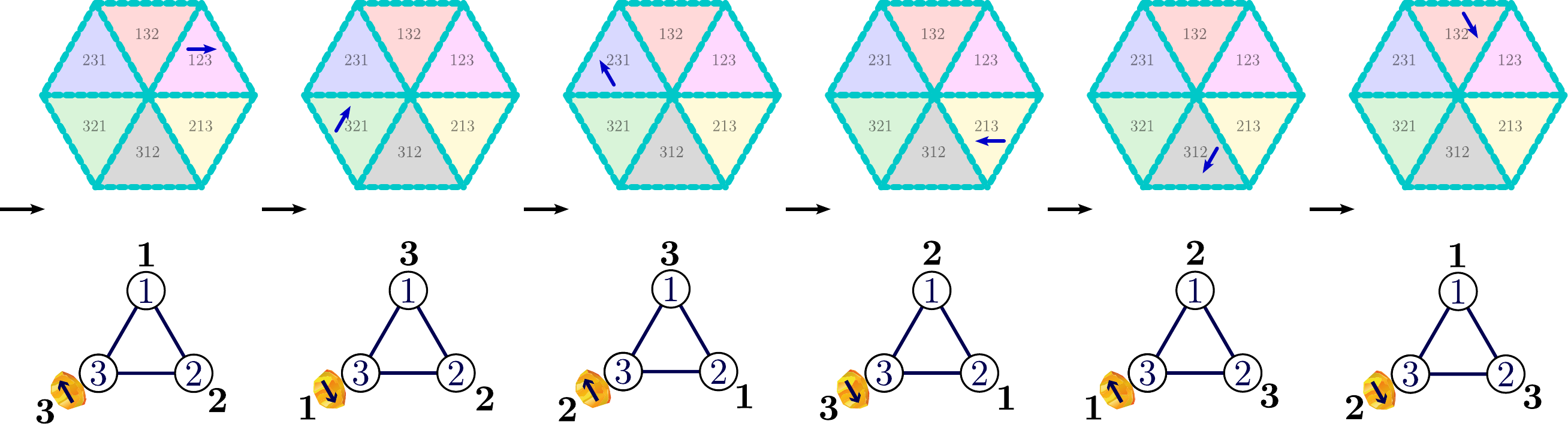}
  \end{center}
\caption{On the top is a bounded combinatorial refraction billiard trajectory in $\affS_3$. On the bottom is the corresponding toric combinatorial refraction billiard trajectory, with each state represented both as an arrow in the torus and as a stone diagram.}\label{fig:K3}
\end{figure}

Let $\mathcal O\subseteq\Xi_n$ be an orbit of $\Theta_{G}$ of size $P$. Let $D\in\mathcal O$ be a stone diagram in this orbit. Consider starting with $D$ and applying $\Theta_{G}$ iteratively $P$ times. If we watch the stone diagrams throughout this process, we will see the replicas move around until eventually settling back into the configuration in which they started. Define the \dfn{winding number} of the replica ${\bf a}$ with respect to $\mathcal O$, denoted $\mathfrak{w}_{\mathcal O}({\bf a})$, to be the net number of clockwise revolutions around $\Cycle_n$ that ${\bf a}$ traverses during this process. In other words, $n\cdot\mathfrak w_{\mathcal O}({\bf a})$ is the net number of clockwise steps that ${\bf a}$ makes during these $P$ iterations of $\Theta_{G}$. (Note that $\mathfrak w_{\mathcal O}({\bf a})$ is negative if ${\bf a}$ moves counterclockwise more than clockwise.) Let us also write $\mathfrak w_{\mathcal O}(a)=\mathfrak w_{\mathcal O}({\bf a})$, where $a$ is the vertex of $G$ corresponding to ${\bf a}$. Since the vertex set of $G$ is $[n]$, we obtain the \dfn{winding vector} $\vec{\mathfrak{w}}_{\mathcal O}=(\mathfrak w_{\mathcal O}(1),\ldots,\mathfrak w_{\mathcal O}(n))$ of the orbit $\mathcal O$. Every time a replica moves one step clockwise, another replica must move one step counterclockwise; this implies that \[\mathfrak{w}_{\mathcal O}(1)+\cdots+\mathfrak{w}_{\mathcal O}(n)=0,\] so $\vec{\mathfrak w}_{\mathcal O}\in Q^\vee$. 

Now let $D=(v,i,\epsilon)\in\mathcal O$, and let $u_0\in\affS_n$ be such that $\overline u_0=v$. Let $\widetilde\Theta_{G}^k(u_0,i,\epsilon)=(u_k,i_k,\epsilon_k)$. As before, let $P$ be the period of $(v,i,\epsilon)$ under $\Theta_{G}$. Then $i_P=i$ and $\epsilon_P=\epsilon$, and we have $\overline{u}_P=v$. In fact, it follows from the definition of the winding vector $\vec{\mathfrak{w}}_{\mathcal O}$ that $\BB u_P=\BB u_0-\vec{\mathfrak{w}}_{\mathcal O}$. By iterating, we find that $\BB u_{\ell P}=\BB u_0-\ell\vec{\mathfrak{w}}_{\mathcal O}$ for every positive integer $\ell$. This means that the sequence $(\BB u_k)_{k\geq 0}$ is bounded (equivalently, periodic) if and only if $\vec{\mathfrak{w}}_{\mathcal O}=0$. We can lift a (closed, toric) billiard trajectory discretized by the sequence $(\Theta_{G}^k(v,i,\epsilon))_{k=0}^P$ to a (piece of a) billiard trajectory discretized by the sequence $(u_k,i_k,\epsilon_k)_{k=0}^P$. It follows that a billiard trajectory discretized by the sequence $(\Theta_{G}^k(v,i,\epsilon))_{k=0}^P$ is contractible if and only if $\vec{\mathfrak{w}}_{\mathcal O}=0$. Thus, we will say that the orbit $\mathcal O$ is \dfn{contractible} if $\vec{\mathfrak{w}}_{\mathcal O}=0$. We deduce the following proposition, which provides a purely combinatorial way to determine if a graph is ensnaring, expelling, or neither. 

\begin{proposition}\label{prop:main}
Let $G=([n],E)$ be a simple graph. Then $G$ is ensnaring if and only if all of the winding vectors of orbits of $\Theta_G$ are $0$. On the other hand, $G$ is expelling if and only if none of the winding vectors of orbits of $\Theta_G$ is $0$. 
\end{proposition}

Whenever we consider the stone diagram associated to a triple $(v,i,\epsilon)\in\Xi_n$, we also consider an associated \dfn{coin diagram}, which consists of the graph $G$ together with a coin placed on the vertex whose replica sits on the same vertex as the stone (that is, we place the coin on the vertex $v^{-1}(i+\frac{1}{2}(1-\epsilon))$). Whenever we apply $\Theta_{G}$, the coin either moves to an adjacent vertex in $G$ or does not move at all. To be more precise, suppose the stone is sitting on vertex $j$ and is pointing toward vertex $j+\epsilon$. Let ${\bf a}$ and ${\bf b}$ be the replicas sitting on $j$ and $j+\epsilon$, respectively. When we apply $\Theta_{G}$, we swap the replicas ${\bf a}$ and ${\bf b}$. If $a$ and $b$ are adjacent in $G$, then the coin moves from $a$ to $b$; otherwise, the coin stays put on vertex $a$. See \cref{fig:coin_example}. 

\begin{figure}[]
  \begin{center}
  \includegraphics[width=0.965\linewidth]{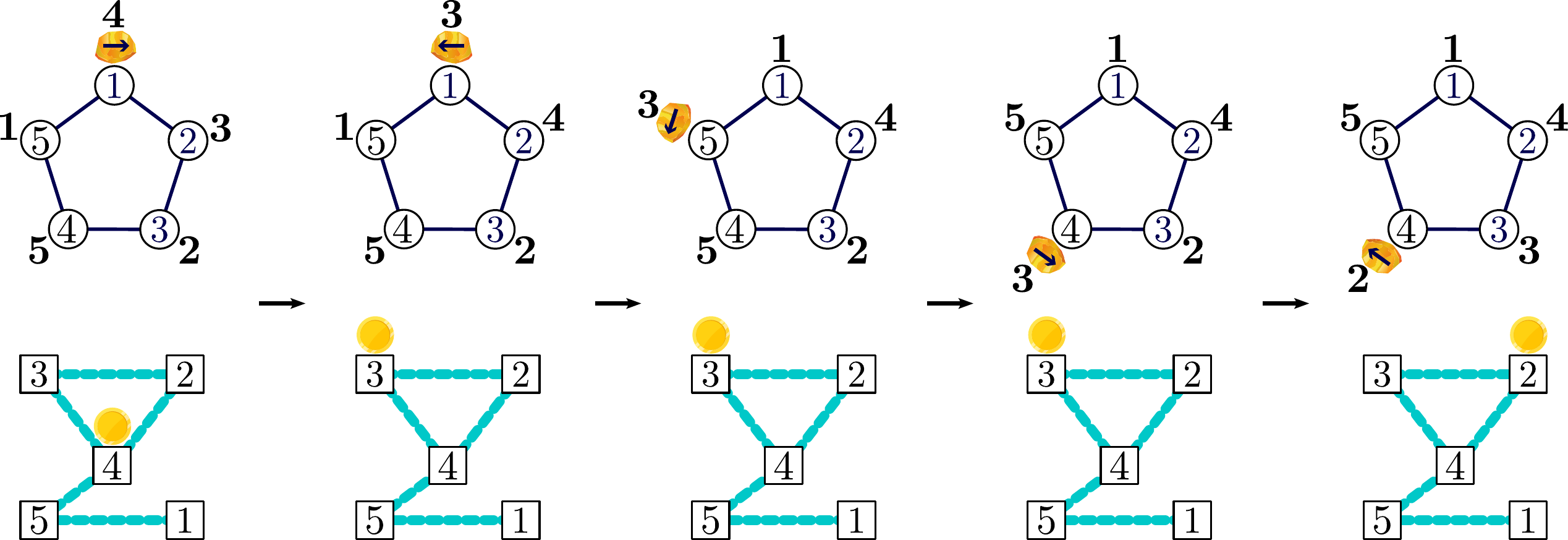}
  \end{center}
\caption{Four applications of the map $\Theta_{G}$. At each step, we represent a triple in $\Xi_5$ as a stone diagram with the associated coin diagram drawn below. }\label{fig:coin_example} 
\end{figure}

\section{Characterization of Expelling Graphs}\label{sec:expelling}
In this section, we prove the following theorem. 

\begin{theorem}\label{thm:expelling}
A graph is expelling if and only if it is bipartite. 
\end{theorem} 

We split the proof into two propositions. 

\begin{proposition}\label{prop:expelling}
Every bipartite graph is expelling. 
\end{proposition}
\begin{proof}
Let $G$ be a bipartite graph, and let $X\sqcup Y$ be a bipartition of its vertex set. Consider an orbit $\mathcal O$ of $\Theta_{G}$ starting with a stone diagram $D$. Without loss of generality, assume the stone points clockwise in $D$. In $D$, the stone coexists with the replica ${\bf x}$ of some vertex $x$; without loss of generality, assume $x\in X$. If $x$ has no neighbors in $G$, then the coin never leaves $x$, so it follows that the winding number of $x$ is nonzero. Now suppose $x$ has at least one neighbor in $G$. At some time during the orbit, the coin must move from $x$ to a vertex $y\in Y$. 

Whenever the coin crosses an edge, the stone reverses its orientation. This implies that the stone points clockwise whenever it sits on a vertex in $X$ and points counterclockwise whenever it sits on a vertex in $Y$. In particular, whenever the replicas ${\bf x}$ and ${\bf y}$ swap places, ${\bf x}$ must move clockwise while ${\bf y}$ moves counterclockwise. This immediately implies that $x$ and $y$ have different winding numbers, so the winding vector of $\mathcal O$ cannot be $0$. 
\end{proof}

\begin{proposition}\label{prop:expelling2}
Every expelling graph is bipartite. 
\end{proposition} 

\begin{proof}
Let $G$ be a graph that is not bipartite. Then $G$ has an induced odd cycle $C$. Let $a_1,\ldots,a_{2m+1}$ be the vertices of $C$ listed in the order they appear around the cycle. We want to construct a periodic orbit.
Say a stone diagram ${(w,i,\epsilon)\in\Xi_n}$ is \dfn{pleasant} if $w(a_1)=0$ and we have the set equalities \[\{w(a_{2j+1}):1\leq j\leq m\}=\{-\epsilon j:1\leq j\leq m\}\quad\text{and}\quad\{w(a_{2j}):1\leq j\leq m\}=\{\epsilon j:1\leq j\leq m\},\] where we view $w$ as a bijection from $[n]$ to $\ZZ/n\ZZ$. 

Start with a pleasant stone diagram $D=(w,i,\epsilon)$, and consider iteratively applying $\Theta_{G}$. The stone travels with ${\bf a}_1$ until it hits ${\bf a}_2$. The stone then reverses direction and travels with ${\bf a}_2$ until hitting ${\bf a}_3$. Then it reverses direction again and travels with ${\bf a}_3$ until hitting ${\bf a}_4$. This process continues until eventually the stone travels with ${\bf a}_{2m+1}$ until hitting ${\bf a}_1$. At this moment, the replicas ${\bf a}_{2j}$ for $1\leq j\leq m$ occupy the positions $-\epsilon,-2\epsilon,\ldots,-m\epsilon$, and the stone coexists with ${\bf a}_1$ and has orientation $-\epsilon$. If we apply $\Theta_{G}$ a few more times, then the stone will travel with ${\bf a}_1$ back to position $0$. Let $D'$ be the stone diagram at this moment when the stone and ${\bf a}_1$ reach position $0$. It is straightforward to show that $D'$ is pleasant. Note that throughout this process starting at $D$ and ending at $D'$, there is never a moment when two replicas swap along the edge $\{m,m+1\}$ of $\Cycle_n$. (See \cref{fig:pleasant}.) We can now repeat this same argument starting at $D'$; we will eventually reach another pleasant stone diagram $D''$ through a process in which no replicas ever swap along the edge $\{m,m+1\}$. Continuing in this manner, we find that we will never swap two replicas along the edge $\{m,m+1\}$. This implies that the winding vector of the orbit of $\Theta_{G}$ containing $D$ is $0$. Hence, $G$ is not expelling. 
\end{proof} 

\begin{example}
Let us illustrate the proof of \cref{prop:expelling2} with the graph 
\[G=\begin{array}{l}
\includegraphics[height=1.653cm]{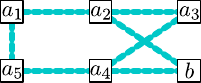}\end{array}.\] Then $n=6$ and $m=2$. Let $D=(w,6,1)$ be the stone diagram at the top left of \cref{fig:pleasant}. Then $D$ is pleasant because $w(a_1)=0$ and we have $\{w(a_3),w(a_5)\}=\{4,5\}=\{-1,-2\}$ and $\{w(a_2),w(a_4)\}=\{1,2\}$ (taking numbers modulo $6$). After $11$ applications of $\Theta_{G}$, we reach the stone diagram at the bottom right of \cref{fig:pleasant}, which is also pleasant. Note that none of these $11$ applications of $\Theta_{G}$ involves swapping two replicas along the edge $\{2,3\}$ of $\Cycle_6$. 
\end{example}

\begin{figure}[]
  \begin{center}
  \includegraphics[width=\linewidth]{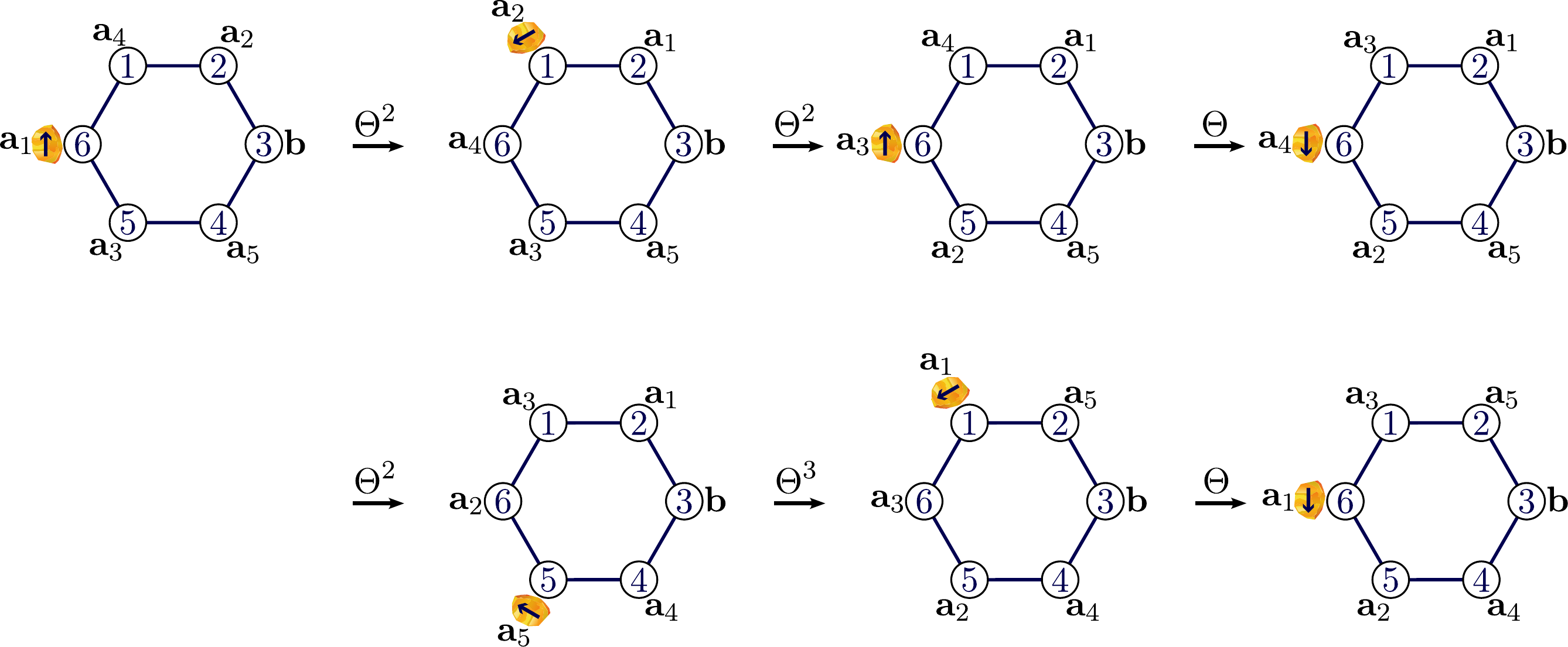}
  \end{center}
\caption{A sequence transforming one pleasant stone diagram into another.   }\label{fig:pleasant}  
\end{figure}

\section{Complete Graphs and Cycles}\label{sec:completecycle}
Our goal in this section is to exhibit two elementary families of ensnaring graphs: complete graphs and odd cycles. In what follows, we denote the $n$-vertex cycle graph by $C_n$; this is to avoid confusion with our notation $\Cycle_n$, which we reserve for stone diagrams. We also use the standard notation $K_n$ for the complete graph with $n$ vertices. 

\begin{theorem}\label{thm:complete}
For every $n\geq 3$, the complete graph $K_n$ is ensnaring.
\end{theorem}

\begin{proof}
Consider an orbit $\mathcal O$ of $\Theta_{K_n}$ containing a stone diagram $D$. Suppose the stone sits on the vertex $i$ of $\Cycle_n$ in the diagram $D$. As we iteratively apply $\Theta_{K_n}$, the stone never leaves position $i$. Thus, in the stone diagrams we encounter, replicas will only swap among positions $i$ and $i\pm 1$. Assuming $n\ge 3$, no replica will ever swap along the edge $\{i-2,i-1\}$. It follows that $\mathcal O$ is contractible. As $\mathcal O$ was arbitrary, $K_n$ is ensnaring. 
\end{proof}

\begin{theorem}\label{thm:cycles}
The cycle graph $C_n$ is ensnaring if and only if $n$ is odd. 
\end{theorem}

\begin{proof}
If $n$ is even, then $C_n$ is bipartite, so it is expelling (and hence, not ensnaring) by \cref{thm:expelling}. 

Suppose $n$ is odd, and let $a$ and $b$ be two adjacent vertices of $C_n$ with replicas $\mathbf a$ and $\mathbf b$. Consider an orbit $\mathcal O$ of $\Theta_{C_n}$. Observe that after the coin moves from $a$ to $b$, the replica $\mathbf a$ will be directly behind the stone, so the stone is guaranteed to reach the replica of the other neighbor of $b$ before it reaches $\mathbf a$. Thus, the coin will next move from $b$ to this other neighbor. Repeating this argument, we conclude the coin only moves in one direction around the cycle. Every time the coin moves, the stone reverses orientation. Because $n$ is odd, after the coin completes a full revolution around $C_n$, the stone's orientation is reversed. Thus, in the orbit $\mathcal O$, the coin must complete an even number of revolutions around $C_n$ so that the stone returns to its original orientation.

Consider the steps in the orbit when $\mathbf a$ and $\mathbf b$ swap. Between two consecutive such steps, the coin completes exactly one revolution around $C_n$, so the stone's orientation in these steps is alternating. As the coin completes an even number of revolutions around $C_n$ in the orbit, $\mathbf a$ swaps an equal number of times in each direction (clockwise or counterclockwise) past $\mathbf b$. This implies $a$ and $b$ have equal winding number. As this is true for any adjacent vertices $a$ and $b$, all vertices have equal winding number. The sum of winding numbers is $0$, so all vertices have winding number $0$. This shows that $\mathcal O$ is contractible. As $\mathcal O$ was arbitrary, $C_n$ is ensnaring.
\end{proof}

\section{Graph Operations on Ensnaring Graphs}\label{sec:wedgeunion}
Given two graphs $G_1$ and $G_2$, a vertex $v_1$ of $G_1$, and a vertex $v_2$ of $G_2$, we define the \dfn{wedge} of $G_1$ and $G_2$ at vertices $v_1$ and $v_2$ to be the graph that results from gluing $v_1$ and $v_2$ together.
In this wedge, the vertices of $G_1$ and of $G_2$ behave independently in a sense: in the stone diagram, the replicas of vertices in $G_1$ behave as if $G_2$ were not there, and vice versa. To formalize this, we need an additional definition. Define the \dfn{conjugate} of a stone diagram $D$ to be the stone diagram $D^\top$ obtained from $D$ by reversing the orientation of the stone (keeping the position of the stone and the positions of the replicas the same). 

\begin{lemma}\label{thm:conjugates}
For any $n$-vertex graph $G$ and any stone diagram $D\in\Xi_n$, we have \[\Theta_{G}(D)^\top=\Theta_{G}^{-1}(D^\top).\]
\end{lemma}

\begin{proof}
Suppose that in $D$, the stone coexists with $\mathbf a$ and points toward $\mathbf b$. Then in $\Theta_{G}(D)$, the stone coexists with either $\mathbf a$ or $\mathbf b$ and points away from the other, so the stone points toward the other one in $\Theta_{G}(D)^\top$. Hence, the permutation of vertex replicas in $\Theta_{G}\left(\Theta_{G}(D)^\top\right)$ differs from that of $\Theta_{G}(D)^\top$ by a swap of $\mathbf a$ and $\mathbf b$, with the stone switching between $\mathbf a$ and $\mathbf b$ and turning around if the edge between $\mathbf a$ and $\mathbf b$ is present, and otherwise staying on $\mathbf a$ without turning. The same can be said about the permutations in $D$ and $\Theta_{G}(D)$. The conjugate diagrams $\Theta_{G}(D)$ and $\Theta_{G}(D)^\top$ differ in stone orientation only, so $D$ and $\Theta_{G}\left(\Theta_{G}(D)^\top\right)$ differ in stone orientation only as well. Thus, $\Theta_{G}(D)^\top=\Theta_{G}^{-1}\left(D^\top\right)$.
\end{proof}

With the concept of conjugate diagrams, we can analyze orbits within a wedge of graphs by looking at each wedged component separately.

\begin{theorem}\label{thm:wedge}
If $G_1$ and $G_2$ are both ensnaring graphs, then any wedge of $G_1$ and $G_2$ is ensnaring.
\end{theorem}

\begin{proof}
Let $G$ be a wedge of $G_1$ and $G_2$, and let $v$ be the shared vertex between $G_1$ and $G_2$ in $G$. Let $\mathcal O$ be an orbit of $\Theta_G$ containing a stone diagram $D_0$, and let $D_i=\Theta_G^i(D_0)$. We will first prove that all vertices in $G_1$, including $v$, have the same winding number with respect to $\mathcal O$. Assume for sake of contradiction that some two vertices $v_1$ and $v_2$ in $G_1$ have different winding numbers with respect to $\mathcal O$. 

Let $n$ be the number of vertices of $G$. Given a stone diagram $D$ in which the stone coexists with a replica of a vertex in $G_1$, let $D^*$ denote the diagram obtained from $D$ by deleting all replicas of vertices not in $G_1$ (we keep the replica of $v$) and then contracting every edge of $\Cycle_n$ that is incident to a vertex that is no longer occupied by a replica. Suppose that for the diagram $D_i$, the coin is on a vertex of $G_1$. If the stone points to another replica of a vertex of $G_1$, then the same transition would occur regardless of the presence of $G_2$, so it follows that $D_{i+1}^*=\Theta_{G_1}(D_i^*)$. 
On the other hand, if the stone points to a replica $\mathbf{u}$ of a vertex $u$ that is not in $G_1$, then either the coin does not move and the stone continues in the same direction, in which case $D_{i+1}^*=D_i^*$, or the coin moves to $u$, in which case $D_{i+1}^*$ is not defined at all. 

Every time the coin moves out of or into $G_1$, it must pass through $v$. While the coin is not in $G_1$, the order of the $G_1$ replicas cannot change, as every time two replicas swap, one of them is outside $G_1$. Thus, if the coin moves out of $G_1$ in the transition from $D_i$ to $D_{i+1}$ and first moves back to $G_1$ in the transition from $D_{j-1}$ to $D_j$, then we know that the order of replicas in $D_i^*$ and $D_j^*$ is the same and that the stone coexists with the replica of $v$ in both. In this case, either $D_i^*=D_j^*$ or $D_i^*={D_j^*}^\top$. 
 
Consider $\ell\geq 1$. Let $j_1<\cdots<j_r$ be the indices $j$ between $0$ and $\ell$ such that $D_j^*$ is defined. For ${k\in\{1,2\}}$, let $\omega_k(\ell)$ be the net number of clockwise steps that the replica ${\bf v}_k$ of $v_k$ takes as we traverse the sequence $D_0,D_1,\ldots,D_\ell$ (so $\omega_k(i)$ is negative if ${\bf v}_k$ moves a net distance counterclockwise). Similarly, let $\omega_k^*(\ell)$ be the net number of clockwise steps that the replica ${\bf v}_k$ takes as we traverse the sequence $D_{j_1}^*,D_{j_2}^*,\ldots,D_{j_r}^*$.  By \cref{thm:conjugates}, for all $i$ such that $D_i^*$ is defined, either $D_i^*$ or ${D_i^*}^\top$ is in the orbit of $\Theta_{G_1}$ containing $D_0^*$, which is finite. As $G_1$ is ensnaring, the orbit of $D_0^*$ is contractible, so $|\omega_1^*(\ell)-\omega_2^*(\ell)|$, and hence also $|\omega_1(\ell)-\omega_2(\ell)|$, is bounded above by a quantity independent of $\ell$. However, our assumption on $v_1$ and $v_2$ guarantees that we can make $|\omega_1(\ell)-\omega_2(\ell)|$ arbitrarily large by choosing $\ell$ large enough; this is a contradiction. 

The preceding argument shows that all vertices in $G_1$ have the same winding number. A similar argument shows that all vertices in $G_2$ have the same winding number. Hence, every vertex has the same winding number as $v$. Since the sum of winding numbers is $0$, every vertex has winding number $0$. As $\mathcal O$ was arbitrary, $G$ is ensnaring. 
\end{proof} 

The converse of \cref{thm:wedge} is false. Indeed, in \cref{sec:wedgecomplete}, we will see examples of graphs $G_1$ and $G_2$ such that $G_1$ is ensnaring, $G_2$ is not ensnaring, and every wedge of $G_1$ and $G_2$ is ensnaring. (For example, one can take $G_1$ to be a complete graph with at least $3$ vertices and $G_2$ to be a tree.) However, we expect the following statement to hold.

\begin{conjecture}\label{conj:wedge_nonensnaring} 
If two graphs $G_1$ and $G_2$ are both not ensnaring, then any wedge of $G_1$ and $G_2$ is also not ensnaring.
\end{conjecture}

Another natural graph operation besides taking wedges is taking disjoint unions. To understand when the disjoint union of two graphs is ensnaring, we will need to distinguish two different classes of ensnaring graphs. Just as we defined the winding number of a replica with respect to an orbit $\mathcal O$ of $\Theta_{G}$, we can define the winding number of the stone with respect to $\mathcal O$, which is just the net number of clockwise revolutions that stone makes throughout $\mathcal O$. 

\begin{definition}\label{def:revolutionary}
An ensnaring graph $G$ is \dfn{revolutionary} if there exists an orbit of $\Theta_G$ for which the stone has a nonzero winding number. We make the convention that a graph with $1$ vertex is revolutionary. 
\end{definition}

\begin{theorem}\label{thm:union}
The disjoint union of two graphs $G_1$ and $G_2$ is ensnaring if and only if both $G_1$ and $G_2$ are ensnaring but not revolutionary.
\end{theorem}

\begin{proof}
Let $G$ be the disjoint union of $G_1$ and $G_2$. Let $n=n_1+n_2$, where $n_i$ is the number of vertices of $G_i$. Consider an orbit of $\Theta_G$. Assume without loss of generality that the coin starts on a vertex of $G_1$. The coin always stays within $G_1$, so the stone always moves past replicas of vertices in $G_2$ when it sees them. In particular, if $n_1=1$, then this orbit is not contractible, so $G$ is not ensnaring. Thus, we may asssume $n_1>1$. 

Given a stone diagram $D$ in which the stone coincides with the replica of a vertex in $G_1$, define $D^*$ (as before) to be the diagram obtained from $D$ by deleting all replicas of vertices in $G_2$ and then contracting every edge of $\Cycle_n$ that is incident to a vertex that is no longer occupied by a replica. If $D^*\neq\Theta_{G}(D)^*$, then $\Theta_{G}(D)^*=\Theta_{G_1}(D^*)$. If $G_1$ is not ensnaring, we can find some initial stone diagram $D_0$ in $\mathcal O$ such that $D_0^*$ has a non-contractible orbit in which some two vertices have different winding numbers. This difference in winding numbers is unaffected by the vertex replicas of $G_2$, so in the orbit of $D_0$, the two vertices will still have different winding numbers. Thus, if $G_1$ is not ensnaring, neither is the disjoint union.

Now suppose $G_1$ is ensnaring. The preceding argument shows that in each orbit of $\Theta_G$ in which the coin stays on $G_1$, all the replicas of vertices in $G_1$ have the same winding number. Consider an orbit $\mathcal O$ of $\Theta_G$ starting with some diagram $D_0$ with the coin on $G_1$, and let $m_1$ be the common winding number of the replicas of the vertices in $G_1$ with respect to $\mathcal O$. Let $m$ be such that the winding number of the stone with respect to $\mathcal O$ is $m+m_1$. The orbit of $\Theta_{G_1}$ containing $D_0^*$, which we call $\mathcal O^*$, is contractible, so $m=0$ if and only if the stone's winding number with respect to $\mathcal O^*$ is $0$ (as the difference in winding numbers of the stone and any replica of $G_1$ is not affected by the presence of $G_2$).

Consider any replica $\mathbf{b}$ of a vertex in $G_2$. Every $n-1$ steps the stone takes in one direction, it moves past $\mathbf{b}$ once, so the winding number of $\mathbf{b}$ with respect to $\mathcal O$ is $-\frac 1{n-1}$ times the stone's winding number with respect to $\mathcal O$. Thus, all replicas of vertices in $G_2$ have the same winding number $-\frac 1{n-1}(m_1+m)$ with respect to $\mathcal O$. Let $m_2=-\frac 1{n-1}(m_1+m)$.

Using the fact that the sum of all winding numbers is $0$, we can solve for all of them. Indeed, \[0=n_1m_1 + n_2m_2 = n_1m_1 - \frac{n_2}{n-1}(m_1+m) = \left(n_1-\frac{n_2}{n-1}\right)m_1 - \frac{n_2}{n-1}m.\] As $n_2\neq 0$ and $n_1-\frac{n_2}{n-1}>n_1-1>0$, we conclude that $m_1=0$ if and only if $m=0$. Thus, $\mathcal O$ is contractible if and only if the stone's winding number with respect to $\mathcal O^*$ is $0$. It follows that $G_1$ is ensnaring but not revolutionary if and only if all orbits in which the coin sits on $G_1$ are contractible. Repeating the above argument with $G_2$ in place of $G_1$ finishes the proof.  
\end{proof}

\cref{thm:union} raises the question: which ensnaring graphs are revolutionary? Clearly, $K_n$ is not revolutionary for $n\ge 3$, as the stone never moves. On the other hand, we can show that most odd cycles are revolutionary.

\begin{theorem}\label{thm:revcycles}
For every odd $n\ge 5$, the cycle graph $C_n$ is revolutionary.
\end{theorem}

\begin{proof}
Assume $n\geq 5$ is odd, and let the vertices of $C_n$ be $1,\ldots,n$, listed in clockwise order. Let $D$ be the stone diagram corresponding to the triple $(\mathbbm{1},1,1)$, where $\mathbbm{1}$ is the identity permutation. In $D$, each replica ${\bf i}$ sits on vertex $i$, and the stone sits on vertex $1$ and points clockwise. Consider the orbit of $\Theta_{C_n}$ starting at $D$. After one step, $\mathbf{1}$ and $\mathbf{2}$ swap, and the stone turns around. Then, the stone carries $\mathbf{2}$ around the whole diagram until it sees $\mathbf{3}$, at which point $\mathbf{2}$ and $\mathbf{3}$ swap and the stone turns around again. The stone diagram at this point in time is $\Theta_{C_n}^{n-1}(D)$; it can be obtained from $D$ by moving every replica one space clockwise and moving the stone two steps clockwise (see \cref{fig:revcycles_example}). During these steps, the stone made a net total of $n-3$ steps counterclockwise. Using the cyclic symmetry of $C_n$, we see that $n-1$ more applications of $\Theta_{C_n}$ will result in the stone moving a net total of $n-3$ steps counterclockwise again. In general, if we start with $D$ and apply $\Theta_{C_n}$ a total of $\ell(n-1)$ times for some positive integer $\ell$, then the stone will move a net total of $\ell(n-3)$ steps counterclockwise. Because $n-3>0$, this implies that the winding number of the stone is nonzero in this orbit. We already know by \cref{thm:cycles} that $C_n$ is ensnaring, so it is revolutionary.  
\end{proof} 

\begin{figure}[]
  \begin{center}
  \includegraphics[width=0.965\linewidth]{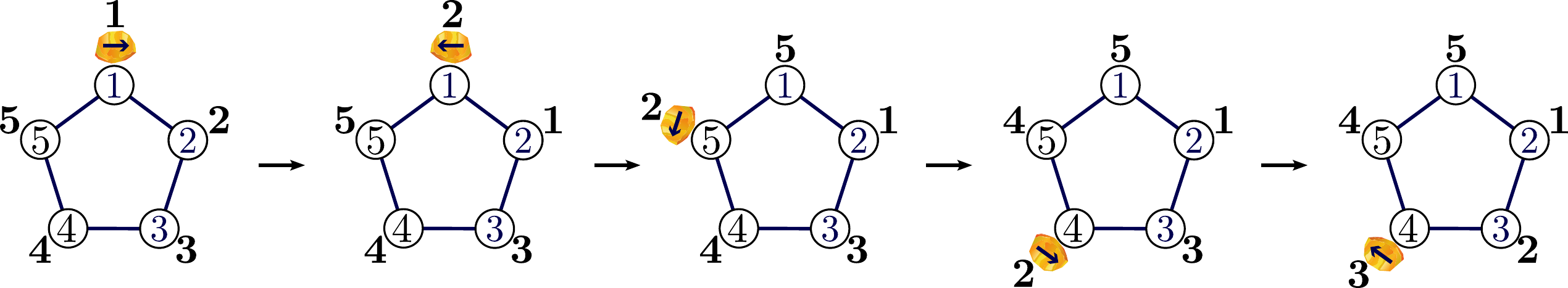}
  \end{center}
\caption{An illustration of the proof of \cref{thm:revcycles} with $n=5$.  }\label{fig:revcycles_example} 
\end{figure}

\section{Wedges Involving Complete Graphs}\label{sec:wedgecomplete}

The orbits in a complete graph are very simple, so it makes sense to consider the wedge of a complete graph with another graph.

\begin{proposition}\label{thm:wedgecomplete}
Let $H$ be a graph. Let $G$ be a wedge of $H$ with a complete graph $K_m$ for $m\geq 3$. Let $v_1,\ldots,v_{m}$ be the vertices of $K_m$, where $v_1$ is the vertex identified with a vertex of $H$ in $G$. If $D_0$ is a stone diagram in which the stone coexists with the replica ${\bf v}_j$ of a vertex $v_j$ with $j\geq 2$, then the orbit of $\Theta_G$ containing $D_0$ is contractible.
\end{proposition} 

\begin{proof}
Consider iteratively applying $\Theta_G$, starting with $D_0$. Let $D_i=\Theta_G^i(D_0)$. If the coin only stays on the vertices of $K_m$, then since $K_m$ is ensnaring (\cref{thm:complete}) and not revolutionary, the proof of \cref{thm:wedge} tells us that the resulting orbit is contractible. Thus, we can assume that at some point, the coin moves from a vertex $u$ of $H$ to $v_1$ and then to some $v_j$ with $j\geq 2$.

When the coin moves onto $v_j$, the replica $\mathbf v_1$ of $v_1$ is behind the stone. As $m\ge 3$, the stone will reach some replica $\mathbf v_k$ with $k\geq 2$ before it reaches $\mathbf v_1$ again. Once the stone does reach $\mathbf{v}_k$, it will turn around and swap $\mathbf v_j$ with $\mathbf v_k$. Suppose this occurs between stone diagrams $D_r$ and $D_{r+1}$.

Notice that $D_r$ and $D_{r+1}$ are almost conjugates; the only additional difference is that $\mathbf v_j$ and $\mathbf v_k$ are swapped. However, $\mathbf v_j$ and $\mathbf v_k$ are symmetric in the graph $G$, so whether they are swapped does not affect the stone diagram's behavior. By an argument similar to the one used to prove \cref{thm:conjugates}, we find that $D_{r+2}$ and $D_{r-1}^\top$ will also only differ by swapping $\mathbf v_j$ and $\mathbf v_k$. Likewise, for all $i$, the diagrams $D_{r+i}$ and $D_{r-i+1}^\top$ only differ by swapping $\mathbf v_j$ and $\mathbf v_k$. Thus, if two replicas swap between $D_{r-i}$ and $D_{r-i+1}$ for some $i$, this swap is undone between $D_{r+i}$ and $D_{r+i+1}$, negating the effect on the winding number of any replica that is not $\mathbf v_j$ or $\mathbf v_k$. It remains to show $\mathbf v_j$ and $\mathbf v_k$ also have winding number~$0$.

Since $K_m$ is ensnaring, the proof of \cref{thm:wedge} tells us that $\mathbf v_1,\ldots,\mathbf{v}_{m}$ all have the same winding number. We already know that $\mathbf v_1$ has winding number $0$. Thus, $\mathbf v_j$ and $\mathbf v_k$ also have winding number $0$, so the orbit containing $D_0$ is contractible.
\end{proof}

If $G$ is a tree, then it follows from \cite[Lemma~5.1]{ADS} that in any orbit of $\Theta_{G}$, every replica swaps past every other replica. This allows us to prove the following corollary. 

\begin{corollary}\label{cor:completetree}
For $m\ge 3$, the wedge of the complete graph $K_m$ with any tree is ensnaring.
\end{corollary}

\begin{proof}
Let $v_1,\ldots,v_m$ be the vertices of a complete graph $K_m$. Let $G$ be a wedge of $K_m$ with a tree $T$, where $v_1$ is the vertex of $K_m$ identified with a vertex of $T$. Consider traversing an orbit $\mathcal O$ of $\Theta_G$. We aim to show that $\mathcal O$ is contractible. By \cref{thm:wedgecomplete}, we just need to show that there is a time when the coin sits on a vertex $v_i$ with $i\geq 2$. Suppose instead that the coin never leaves $T$. Let $T'$ be the induced subgraph of $G$ formed by $T$ and the vertex $v_2$, and note that $T'$ is a tree. Since the coin never leaves $T$, it of course never leaves $T'$; this means that the replicas move around in the orbit of $\Theta_G$ in the same way that they would in an orbit of $\Theta_{T'}$. Using \cite[Lemma~5.1]{ADS}, we find that there is a time when $\mathbf{v}_1$ swaps with $\mathbf{v}_2$. But then the coin must move from $v_1$ to $v_2$, which is a contradiction. 
\end{proof}

Notably, trees are not ensnaring; in fact, since they are bipartite, \cref{thm:expelling} tells us they are expelling. Thus, this corollary shows that the converse of \cref{thm:wedge} is false. We can extend this result even further in the following theorem.

\begin{theorem}\label{thm:multitree}
Fix $m\ge 3$. Let $G$ be a graph obtained by wedging a (possibly one-vertex) tree onto each vertex of $K_m$. Then $G$ is ensnaring.
\end{theorem}

\begin{proof}
Let $v_1$, $\dots$, $v_m$ be the vertices of $K_m$, and let $T_i$ be the tree wedged onto $v_i$. Any orbit in which the coin never leaves the vertices $v_1,\ldots,v_m$ is contractible, as the stone never moves. Now let $\mathcal O$ be an orbit in which the coin moves off $v_a$ and into $T_a\setminus \{v_a\}$ for some $a$. The same argument as in the proof of \cref{cor:completetree} (invoking \cite[Lemma~5.1]{ADS}) shows that the coin must eventually move onto some vertex $v_b$ for $b\neq a$.

Let $D_{a,1}$ be the stone diagram right before the coin first moves off $v_a$ into $T_a\setminus\{v_a\}$, and let $D_{a,1}'$ be the stone diagram right after it first moves back onto $v_a$. Then \cite[Lemma~5.1]{ADS} tells us that $D_{a,1}'$ is simply a rotation of $D_{a,1}$. The same holds every time the coin moves into and out of some $T_i\setminus \{v_i\}$. Thus, if we ignore any time the coin is not on a $v_i$, then the coin's movement behaves as it would on a $K_m$. In particular, it moves in a $3$-cycle $v_a\to v_b\to v_c\to v_a$ for some $c$. The stone's orientation reverses with each of these moves, so the direction in which the replica $\mathbf v_a$ swaps past the replica $\mathbf v_b$ alternates each time. Thus, $v_a$ and $v_b$ have the same winding number with respect to $\mathcal O$. Similarly, $v_b$ and $v_c$ have the same winding number with respect to $\mathcal O$. 

Consider two repetitions of the coin's $3$-cycle: $v_a\to v_b\to v_c\to v_a\to v_b\to v_c\to v_a$. Let $D$ be the stone diagram right before the coin first moves from $v_b$ to $v_c$, and let $D'$ be the stone diagram right after it first moves from $v_a$ to $v_b$. Then, similar to the proof of \cref{thm:wedgecomplete}, we know $D'$ can be obtained from taking $D$, swapping $\textbf v_a$ and $\textbf v_c$, and then taking a rotation. The rotation does not affect the stone's behavior, and neither $\textbf v_a$ nor $\textbf v_c$ is in $T_b$, so as in the proof of \cref{thm:conjugates}, the swaps that occur the first time the coin is in $T_b$ (after it moves onto $v_b$ but before it moves onto $v_c$) are undone by those that occur the second time the coin is in $T_b$, except for those involving $\textbf v_a$ or $\textbf v_c$.

A similar statement holds for the first and second times the coin is in $T_a$ and $T_c$, so the winding number of any vertex which is not $v_a$, $v_b$, or $v_c$ is $0$. Since $v_a$, $v_b$, and $v_c$ have the winding number and the sum of all winding numbers is $0$, every vertex has winding number $0$. It follows that $\mathcal O$ is contractible. As $\mathcal O$ was arbitrary, we conclude that $G$ is ensnaring. 
\end{proof}

\section{A Non-Ensnaring Configuration}\label{sec:nonensnaring}

The same idea used to prove \cref{thm:wedgecomplete} can also be used to prove certain graphs are not ensnaring. \cref{fig:configuration} illustrates the configuration described in the next theorem.

\begin{figure}[ht]
  \begin{center}
  \includegraphics[height=2cm]{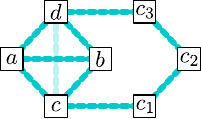}
  \end{center}
\caption{A schematic illustration of the configuration described in \cref{thm:local_not_ensnaring}. The edge between $c$ and $d$ is depicted as translucent to indicate that it may or may not be present.  }\label{fig:configuration}
\end{figure}

\begin{theorem}\label{thm:local_not_ensnaring}
Suppose a graph $G$ has vertices $a$, $b$, $c$, and $d$ with the following properties:
\begin{enumerate}
    \item $a$ has exactly $3$ neighbors, which are $b$, $c$, and $d$.
    \item $b$ has exactly $3$ neighbors, which are $a$, $c$, and $d$.
    \item There exists a path from $c$ to $d$ with at least two edges that does not pass through $a$ or $b$. (The edge between $c$ and $d$ may or may not be present.)
\end{enumerate}
Then $G$ is not ensnaring.
\end{theorem}

\begin{proof}
Let $c=c_0,c_1,\ldots,c_m=d$ be the vertices along a path from $c$ to $d$, where $m\geq 2$. Without loss of generality, assume this path is of minimal length; in particular, there are no other edges between vertices in this path other than the edge in the path itself, except possibly between $c$ and~$d$. 

Consider a stone diagram $D$ in which the vertex replicas are in clockwise order \[\mathbf d=\mathbf c_m,\mathbf a,\mathbf c_{m-2},\mathbf c_{m-4},\ldots,\mathbf c=\mathbf c_0,\ldots,\mathbf c_{m-3},\mathbf c_{m-1},\mathbf b,\] followed by the remaining replicas, with the stone starting at $\mathbf c_0$ and pointing away from $\mathbf c_1$. We know $c_0$ is not adjacent to any $c_i$ with $2\leq i\leq m-1$, so the stone will first swap $\mathbf c_0$ with one of $\mathbf a$ or $\mathbf b$, turning around while doing so. After this swap, the stone has $\mathbf c_0$ behind it, and it is on $\mathbf a$ or $\mathbf b$. It will encounter the other one of these two replicas before $\mathbf c$ or $\mathbf d$, so the stone will next swap $\mathbf a$ and $\mathbf b$, turning around while doing so. Observe that $\mathbf a$ swaps clockwise past $\mathbf b$, regardless of whether the stone is on $\bf a$ or $\bf b$. Similar to the proof of \cref{thm:wedgecomplete}, the diagrams before and after this swap are almost conjugates, differing by just a swap of $\mathbf a$ and $\mathbf b$, which correspond to symmetric vertices. Thus, we will eventually reach the stone diagram that is obtained from $D^\top$ by swapping $\mathbf a$ and $\mathbf b$; call this diagram $D_1$. Note that the stone coexists with $\mathbf c_0$ 
and points toward $\mathbf c_1$ in $D_1$. The order of the $\mathbf c_i$'s
in $D_1$ is such that when the stone turns around after swapping $\mathbf c_{j-1}$ and $\mathbf c_j$ for some $j<m-1$, the stone will then encounter only $\mathbf c_i$'s until it reaches $\mathbf c_{j+1}$. This is even true when $j=m-1$, except the stone will also move past $\mathbf b$ (note that $b$ and $c_{m-1}$ cannot be neighbors). Once the stone reaches $\mathbf c_m=d$, it will be pointing clockwise toward $\mathbf b$, so these two replicas will swap and the stone will turn counterclockwise. The stone will then encounter $\mathbf a$ before any other neighbor of $\mathbf b$, so $\mathbf a$ and $\mathbf b$ will swap. Notably, in this swap, $\mathbf a$ swaps clockwise past $\mathbf b$.

After $\mathbf a$ and $\mathbf b$ swap, the previous steps will reverse until we return to the stone diagram obtained from $D_1^\top$ by swapping $\mathbf a$ and $\mathbf b$, which is just $D$. This completes an orbit in which $\mathbf a$ only swapped clockwise past $\mathbf b$. Thus, $\mathbf a$ and $\mathbf b$ have unequal winding numbers in this orbit, so this orbit is not contractible, and $G$ is not ensnaring. 
\end{proof}

\cref{thm:local_not_ensnaring} highlights the fact that a graph being non-ensnaring is a very local property, as certain configurations in a graph can ``trap'' orbits and force them to be not contractible. While local conditions as in \cref{thm:wedgecomplete} can ensure a particular orbit is contractible, they cannot ensure a graph is ensnaring because of the possibility of ``trapped'' non-contractible orbits.

\section{Connected Components of the Complement Graph}\label{sec:connected_complement} 

The coin is a useful tool for analyzing orbits in sparse graphs, as it moves infrequently and its movement is very constrained. When looking at orbits in a dense graph $G$, it makes more sense to focus on when the coin does not move, rather than when it does. 

Let $\overline G$ denote the complement of $G$. That is, $\overline G$ is the graph with the same vertex set as $G$ such that two vertices are adjacent in $\overline G$ if and only if they are not adjacent in $G$. 

Consider a stone diagram $D$ in which the stone coexists with a replica ${\bf a}$ and points toward a replica ${\bf b}$. Let us say $D$ is \dfn{prerefractive} if the corresponding vertices $a$ and $b$ of $G$ are not adjacent in $G$. If $D$ is prerefractive, then we define the \dfn{bridged edge} associated to $D$ to be the directed edge $a\to b$; note that $a\to b$ is an orientation of an edge of $\overline G$. If $D$ is not prerefractive, then we define the bridged edge associated to $D$ to be the bridged edge associated to $\Theta_G^{-\ell}(D)$, where $\ell$ is the smallest positive integer such that $\Theta_{G}^{-\ell}(D)$ is prerefractive (if no such $\ell$ exists, then there is no bridged edge associated to $D$). 

\begin{proposition}\label{prop:bridge}
Let $\mathcal O$ be an orbit of $\Theta_G$. There is a connected component $C$ of $\overline G$ such that in every stone diagram in $\mathcal O$, the stone either coexists with the replica of a vertex in $C$ or points toward the replica of a vertex in $C$. 
\end{proposition}

\begin{proof}
If $\mathcal O$ contains no prerefractive stone diagrams, then the proposition is trivial because the stone never moves when we apply $\Theta_G$ to a stone diagram in $\mathcal O$. Now assume there is a prerefractive stone diagram $D$ in $\mathcal O$, and let $a\to b$ be the bridged edge associated to $D$. Let $k$ be the smallest positive integer such that $\Theta_{G}^k(D)$ is prerefractive. In $D$, the stone is on $\mathbf a$ and points toward $\mathbf b$. Let $\mathbf x$ be the replica that sits next to ${\bf b}$ in $D$ and is not ${\bf a}$. We claim that in each of the stone diagrams $\Theta_{G}^i(D)$ for $1\leq i\leq k$, the stone either coexists with ${\bf a}$ or ${\bf b}$ or points toward ${\bf a}$ or ${\bf b}$. Note that this claim implies the proposition. To prove the claim, we consider three cases based on which of $a$ and $b$ are neighbors of $x$ in $G$. 

\textbf{Case 1:} Suppose $a$ and $x$ are not adjacent in $G$. Then in $\Theta_G(D)$, the stone coexists with ${\bf a}$ and points toward ${\bf x}$. In this case, we have $k=1$, and the claim holds. 

\textbf{Case 2:} Suppose $a$ and $x$ are adjacent in $G$ and that $b$ and $x$ are not adjacent in $G$. Then in $\Theta_G(D)$, the stone coexists with ${\bf a}$ and points toward $\mathbf x$. In $\Theta_G^2(D)$, the stone coexists with ${\bf x}$ and points toward ${\bf b}$. In this case, we have $k=2$, and the claim holds. 

\textbf{Case 3:} Suppose both $a$ and $b$ are adjacent to $x$ in $G$. Then in $\Theta_G(D)$, the stone coexists with ${\bf a}$ and points toward $\mathbf x$. In $\Theta_G^2(D)$, the stone coexists with $\mathbf x$ and points toward $\mathbf b$. In $\Theta_G^3(D)$, the stone coexists with ${\bf b}$ and points toward ${\bf a}$. In this case, we have $k=3$, and the claim holds.
\end{proof}

Henceforth, we refer to prerefractive stone diagrams corresponding to the three cases above as \dfn{pre-pivotal}, \dfn{post-pivotal}, and \dfn{flipping}, respectively.

\cref{prop:bridge} tells us that, in some sense, the connected components of $\overline{G}$ behave separately. To formalize this idea, we need the following definition. Given a graph $H$ with $m$ vertices and an integer $n\ge m$, define the \dfn{$n$-vertex complement} of $H$, denoted $\compl_n(H)$, to be the $n$-vertex graph whose complement consists of $H$ together with $n-m$ additional isolated vertices.

\begin{theorem}\label{thm:complcomp}
A graph $G$ with $n$ vertices is ensnaring if and only if the $n$-vertex complement of each connected component of $\overline{G}$ is ensnaring. 
\end{theorem}

\begin{proof}
Let $\Upsilon$ be the set of orbits of $\Theta_G$ that do not contain prerefractive stone diagrams, and let $\Upsilon'$ be the set of orbits of $\Theta_G$ that are not in $\Upsilon$. If we traverse an orbit in $\Upsilon$, then the stone never moves. This shows that all orbits in $\Upsilon$ are contractible. Thus, $G$ is ensnaring if and only if all of the orbits in $\Upsilon'$ are contractible. 

For each $\mathcal O\in\Upsilon'$, it follows from \cref{prop:bridge} that there is a unique connected component $C^{\mathcal O}$ of $\overline G$ such that all of the bridged edges associated to stone diagrams in $\mathcal O$ are in $C^{\mathcal O}$. For each connected component $C$ of $\overline G$, let $\Upsilon_C'=\{\mathcal O\in\Upsilon':C^{\mathcal O}=C\}$. Consider $\mathcal O\in\Upsilon_C'$. \cref{prop:bridge} tells us that in each stone diagram $D\in\mathcal O$, the stone either coexists with the replica of a vertex in $C$ or points toward the replica of a vertex in $C$. This implies that the orbit $\mathcal O$ does not depend on which edges are present between vertices outside of $C$. More precisely, $\Theta_{\compl_n(C)}^i(D)=\Theta_G^i(D)$ for every integer $i$. As $\mathcal O$ was chosen arbitrarily, this shows that $\compl_n(C)$ is ensnaring if and only if all of the orbits in $\Upsilon_C'$ are contractible. As this holds for every connected component $C$ of $\overline G$, the desired result follows. 
\end{proof}

Whether the $n$-vertex complement of a graph is ensnaring depends on $n$, even though the additional vertices all behave the same way. To what extent does $n$ matter? The following conjecture suggests an answer to this question. 

\begin{conjecture}\label{conj:complpar}
Fix an $m$-vertex graph $G$. Then for all integers $n>m$, whether $\compl_n(G)$ is ensnaring depends only on the parity of $n$.
\end{conjecture}


\begin{theorem}\label{thm:complcomplete}
For any integers $n\ge m\ge 2$, the graph $\compl_n(K_m)$ is not ensnaring.
\end{theorem}

\begin{proof}
The result is trivial when $n=m$ since then $\compl_n(K_m)$ consists of $n$ isolated vertices. Thus, we may assume $n>m$. Fix a vertex $x$ in $\compl_n(K_m)$ that is adjacent to all other vertices. 

Let us iteratively apply $\Theta_{\compl_n(K_m)}$ starting with a prerefractive stone diagram $D$. Without loss of generality, assume the stone points clockwise in $D$. Any two vertices that are not adjacent in $\compl_n(K_m)$ have the same set of neighbors. Therefore, no stone diagram is post-pivotal. After one step from a pre-pivotal stone diagram or three steps from a flipping stone diagram, the stone moves a net of one step clockwise and keeps its clockwise orientation. Thus, the stone will eventually point toward $\mathbf x$. When the stone points toward $\mathbf x$, it must be pointing clockwise, so it swaps $\mathbf x$ with its current replica, moving it one step counterclockwise. It then swaps $\mathbf x$ with the replica the stone points toward, moving $\mathbf x$ another step counterclockwise. Thus, $\mathbf x$ only moves counterclockwise, so the orbit of $\Theta_{\compl_n(K_m)}$ containing $D$ is not contractible. Hence, $\compl_n(K_m)$ is not ensnaring.
\end{proof}

The following corollary is immediate from \cref{thm:complcomp,thm:complcomplete}. 

\begin{corollary}\label{cor:complcomplete}
If $\overline G$ has a connected component that is complete and has at least two vertices, then $G$ is not ensnaring.
\end{corollary}

The operations in \cref{sec:wedgeunion} tend to create sparse graphs with large sets of vertices that do not interact. By looking at complements of graphs, we can get ensnaring graphs far denser than most examples we have seen. In what follows, we let $K_{\ell,r}$ denote the complete bipartite graph with partite sets of sizes $\ell$ and $r$. 

\begin{theorem}\label{thm:oddglue}
Let $\ell,r,n$ be integers such that $\ell+r<n$. The graph $\compl_n(K_{\ell, r})$ is ensnaring if and only if $n-\ell-r$ is odd and $\ell$ and $r$ are not both $1$.
\end{theorem}

\begin{proof}
Let $G=\compl_n(K_{\ell, r})$. If $\ell=r=1$, then $K_{\ell,r}=K_2$, so \cref{thm:complcomplete} tells us that $G$ is not ensnaring. Henceforth, assume either $\ell$ or $r$ is not $1$; we may assume without loss of generality that $\ell\geq 2$. Let $m=n-\ell-r$. We can partition the vertices of $G$ into sets $A=\{a_1,\dots,a_\ell\}$, $B=\{b_1,\dots, b_r\}$, and $X=\{x_1,\dots, x_m\}$ so that every edge is present except those connecting $A$ and $B$. 

Suppose first that $m$ is even; we will construct a stone diagram $D_0$ such that the orbit of $\Theta_G$ containing $D_0$ is not contractible. To construct $D_0$, place the replicas on $\Cycle_n$ so that \begin{itemize}
\item the replicas of vertices in $X$ appear in the clockwise cyclic order ${\bf x}_1,\ldots,{\bf x}_m$; 
\item the replicas of vertices in $B$ are all in the clockwise arc from ${\bf x}_m$ to ${\bf x}_1$; 
\item the replicas of vertices in $A\setminus\{a_1\}$ are all in the clockwise arc from ${\bf x}_1$ to ${\bf x}_2$; 
\item ${\bf a}_1$ is one space counterclockwise of ${\bf b}_1$. 
\end{itemize}
Then place the stone so that it coexists with ${\bf a}_1$ and points toward ${\bf b}_1$. Note that $D_0$ is prerefractive and has associated bridged edge $a_1\to b_1$.  

Imagine starting with $D_0$ and iteratively applying $\Theta_G$. Until the stone sees ${\bf x}_1$, the bridged edge will continue to point from $A$ to $B$, and the stone will not turn around. After the stone reaches ${\bf x}_1$, the stone coexists with ${\bf x}_1$ for two steps, and then the stone moves clockwise past ${\bf x}_1$. At this time, the bridged edge reverses orientation, and it continues to point from $B$ to $A$ until the stone reaches $\mathbf x_2$. The stone will then move past ${\bf x}_2,\ldots,{\bf x}_m$ in turn, with the bridged edge reversing orientation each time so that afterwards it points from $A$ to $B$. This process will then repeat. The same argument used in the proof of \cref{thm:complcomplete} tells us that $\mathbf x_1$ only moves counterclockwise, so the orbit is not contractible. 

Now suppose $m$ is odd. Let $\mathcal O$ be an orbit of $\Theta_G$; we will show that $\mathcal O$ is contractible. We may assume there exists a prerefractive stone diagram $D\in\mathcal O$ since otherwise $\mathcal O$ is certainly contractible. Without loss of generality, we can assume the stone coexists with ${\bf a}_1$ and points clockwise toward ${\bf b}_1$ in $D$. Furthermore, we may assume that if we start at the stone and read the replicas clockwise around the diagram, the replicas of the vertices in $X$ appear in the order $\mathbf x_1,\mathbf x_2,\ldots,\mathbf x_m$.

We claim that $\mathcal O$ must contain a post-pivotal stone diagram. Let $k$ be the smallest positive integer such that the stone points toward ${\bf a}_1$ in $\Theta^k(D)$. Let $k'$ be the smallest positive integer such that the stone points toward ${\bf a}_1$ in $\Theta_G^{k+k'}(D)$. If $\Theta_G^k(D)$ is pre-pivotal and no diagram between $\Theta^k(D)$ and $\Theta^{k+k'}(D)$ is post-pivotal, then there are $m$ times while we transition from $\Theta_G^k(D)$ to $\Theta_G^{k+k'}(D)$ during which the bridged edge changes whether it points from $A$ to $B$ or from $B$ to $A$. Since $m$ is odd, the bridged edge ends up in the opposite orientation. Hence, when the stone sees ${\bf a}_2$ again, the resulting stone diagram will be post-pivotal.

Without loss of generality, we can assume that the current stone diagram is post-pivotal: after the stone moves $\mathbf a_1$ past $\mathbf b_1$, it swaps $\mathbf a_1$ and $\mathbf a_2$ (we switch $\ell$ and $r$ if necessary). The diagrams before and after this swap are nearly conjugates; as $a_1$ and $a_2$ are symmetric vertices, similar to the proof of \cref{thm:wedgecomplete}, we know all replicas except $\mathbf a_1$ and $\mathbf a_2$ have winding number $0$.

If, as the stone moves counterclockwise, it encounters any replica of the form $\mathbf a_i$ before $\mathbf x_m$, the resulting stone diagram will be post-pivotal. Otherwise, once the stone moves past $\mathbf x_m$, the bridged edge reverses orientation, so if it encounters any replica of the form $\mathbf b_i$ before $\mathbf x_{m-1}$, the resulting stone diagram will be post-pivotal. This alternating pattern continues.

Observe that the stone begins going counterclockwise with $\mathbf a_1$ right behind it. If the next post-pivotal stone diagram has bridged edge $\mathbf a_1\to \mathbf a_2$ or $\mathbf a_2\to \mathbf a_1$, then it takes two full revolutions of the stone for this to occur. If either $\mathbf a_3$ or $\mathbf b_2$ exists, it is impossible to avoid an earlier post-pivotal stone diagram when the stone passes by this replica in both revolutions. Thus, unless $\ell=2$ and $r=1$, the next post-pivotal stone diagram's bridged edge will not involve both $\mathbf a_1$ and $\mathbf a_2$. Using the argument from \cref{thm:wedgecomplete} again, we see one of $\mathbf a_1$ or $\mathbf a_2$ has winding number $0$. As the sum of all winding numbers is $0$, we conclude the other has winding number $0$ as well, so the orbit is contractible.

It remains to address the case where $\ell=2$ and $r=1$. When $\mathbf a_1$ and $\mathbf a_2$ first swap, $\mathbf a_1$ swaps clockwise past $\mathbf a_2$, and the bridged edge becomes $\mathbf a_2\to \mathbf b_1$. The next time the stone sees $\mathbf a_1$, the bridged edge has reversed orientation $m$ times (once for each $\mathbf x_i$), so it is now $\mathbf b_1\to\mathbf a_2$. The resulting stone diagram is pre-pivotal, so $\mathbf a_1$ and $\mathbf a_2$ do not swap, and the bridged edge pivots back to $\mathbf b_1\to\mathbf a_1$. The stone then completes another revolution counterclockwise before seeing $\mathbf a_2$ again, with the bridged edge reversing orientation another $m$ times to $\mathbf a_1\to\mathbf b_1$. This stone diagram is post-pivotal, so $\mathbf a_1$ swaps counterclockwise past $\mathbf a_2$.

Thus, every time $\mathbf a_1$ swaps past $\mathbf a_2$, the direction it does so alternates. We conclude that both replicas have the same winding number. The remaining winding numbers are all $0$, and the sum of all winding numbers is $0$, so $\mathbf a_1$ and $\mathbf a_2$ have winding number $0$ as well. Thus, the orbit is contractible. This proves that $G$ is ensnaring, as desired.
\end{proof}

\section{Future Directions}\label{sec:future}
\subsection{Ensnaring Graphs}
There are still many questions concerning ensnaring graphs that we have not answered. Of course, it would be wonderful to have a complete characterization of ensnaring graphs, though that seems quite difficult at the moment. Here we list some specific problems that could be more manageable.  

First, let us recall two conjectures that we stated earlier. \cref{conj:wedge_nonensnaring} says that a wedge of two non-ensnaring graphs is nonensnaring. \cref{conj:complpar} says that whether the $n$-vertex complement of a fixed graph is ensnaring depends only on the parity of $n$. Proving either of these statements would provide very useful structural information about the family of ensnaring graphs. 

With an eye toward understanding dense graphs, it is natural to ask when the $n$-vertex complement of a tree is ensnaring. 

\begin{conjecture}\label{conj:compltree}
Let $T$ be an $m$-vertex tree, and let $n>m$ be an even integer. Then $\compl_n(T)$ is ensnaring if and only if $m$ is odd.
\end{conjecture}

\subsection{Reflections and Refractions}\label{subsec:refl}
The article \cite{ADS} considers combinatorial billiard systems that are more general than the ones that we have considered so far because they allow for mirrors, which are a third type of hyperplane besides windows and metalenses. When a beam of light hits a mirror, it reflects. We can also define a toric mirror to be the image of a mirror under the quotient map $\qq$. 

To define a system with reflections and refractions, we fix a partition $E=\Eflect\sqcup\Efract$ of the edge set of our graph $G$ into two subsets. We call the pair $(G;\Eflect\sqcup\Efract)$ a \dfn{materialized graph}. The edges in $\Eflect$ are called \dfn{reflection edges}, while the edges in $\Efract$ are called \dfn{refraction edges}. Then a hyperplane $\HH_{i,j}^k$ is a mirror if $\{i,j\}\in\Eflect$, and it is a metalens if $\{i,j\}\in\Efract$. Thus, the setting that we have considered so far in this article is the one in which $E=\Efract$ (and $\Eflect=\emptyset$). We modify the definition of $\widetilde\Theta_G$ so that 
$\widetilde\Theta_G(w,i,\epsilon)=(w,i+\epsilon,\epsilon)$ if $\HH^{(w,s_i)}$ is a mirror (when the beam of light reflects off of $\HH^{(w,s_i)}$, it stays in the alcove $\BB w$). This leads to the following more general definition of $\Theta_G$:
\[\Theta_G(v,i,\epsilon)=\begin{cases} (\overline s_i v,i+\epsilon,\epsilon) & \mbox{if }\{v^{-1}(i),v^{-1}(i+1)\}\not\in E; \\   (\overline s_iv,i-\epsilon,-\epsilon) & \mbox{if }\{v^{-1}(i),v^{-1}(i+1)\}\in \Efract; \\   (v,i+\epsilon,\epsilon) & \mbox{if }\{v^{-1}(i),v^{-1}(i+1)\}\in \Eflect.  \end{cases}\]
(Note that $\Theta_G$ now also depends on the partition $\Eflect\sqcup\Efract$, although we omit that dependence from the notation.) When we apply $\Theta_G$ to a stone diagram in which the stone sits on $j$ and has orientation $\epsilon$, we look at the replicas sitting on vertices $j$ and $j+\epsilon$. If the vertices corresponding to these two replicas are the endpoints of a reflection edge, then we move the stone from vertex $j$ to vertex $j+\epsilon$ (without changing its orientation), and we do not move any of the replicas; otherwise, we modify the stone diagram as before (depending on whether these two replicas are non-adjacent or are the endpoints of a refraction edge). 

Our decision to neglect reflection edges throughout most of the article was born out of a desire for simplicity. However, we do wish to briefly indicate how some of our results generalize when we have both reflection edges and refraction edges. 
We will omit most of the details.  

Say the materialized graph $(G;\Eflect\sqcup\Efract)$ is \dfn{ensnaring} (respectively, \dfn{expelling}) if all (respectively, none) of the corresponding billiard trajectories are contractible. As before, we can define winding numbers and winding vectors and show that a materialized graph is ensnaring (respectively, expelling) if and only if all (respectively, none) of the winding vectors of the orbits are $0$.   

Two vertices that form the endpoints of a reflection edge must have the same winding number in any orbit of $\Theta_G$ because their replicas can never move past each other. In particular, if $\Efract$ is empty and $G$ is connected, then $(G;\Eflect\sqcup\Efract)$ is ensnaring. 

As a generalization of \cref{prop:expelling}, one can show (using essentially the same proof) that $(G,\Eflect\sqcup\Efract)$ is expelling if $\Efract$ is nonempty and the graph obtained from $G$ by contracting all reflection edges is bipartite. However, the converse of this statement does not hold; this is unlike \cref{prop:expelling}, whose converse is \cref{prop:expelling2}. 

\cref{thm:conjugates}, \cref{thm:wedge}, \cref{thm:union}, and \cref{thm:wedgecomplete} generalize readily to the setting with reflection edges and refraction edges. Furthermore, \cite[Lemma~5.1]{ADS} applies to all materialized trees, not just those with only refraction edges. Using this, we can show that the wedge of any complete graph with only refraction edges with an arbitrary materialized tree is ensnaring. The proof is essentially the same as that of \cref{cor:completetree}. 

Finally, suppose $G$ is a cycle graph. Generalizing \cref{thm:cycles}, we can show that $(G;\Eflect\sqcup\Efract)$ is ensnaring if $|\Efract|$ is $0$ or odd and is expelling otherwise. Indeed, if $|\Efract|$ is even, this follows from some of the preceding remarks. Now suppose $|\Efract|$ is odd. The proof of \cref{thm:cycles} shows that in every orbit of $\Theta_G$, any two vertices connected by a refractive edge have the same winding number. Furthermore, any two vertices connected by a reflective edge have the same winding number. Thus, all the vertices have the same winding number, which must be $0$. 

In general, it would be interesting to gain a better understanding about which materialized graphs are expelling and which are ensnaring. Can we completely characterize expelling materialized graphs, as we did in \cref{thm:expelling} in the case where $\Eflect=\emptyset$?  

\subsection{Orbit Sizes} 
The article \cite{ADS} by Adams, Defant, and Striker introduced refractive toric promotion as a novel combinatorial dynamical system modeling a combinatorial billiard system. The primary goal of that article was to find materialized graphs $G$ for which $\Theta_G$ has well-behaved orbit sizes. While we had quite different topological motivations in this article, our investigations caused us to stumble upon a new class of examples that appear to have fascinating orbit structures. 

\begin{conjecture}
Let $P_m$ be a path graph with $m$ vertices, and let $G=\compl_n(P_m)$ be the $n$-vertex complement of $P_m$. If $m$ is odd and $n$ is even, then every orbit of $\Theta_G$ has size $6$ or $6(mn-n-3m+4)$.
\end{conjecture} 

\subsection{Other Types} 
The article \cite{ADS} discusses how to define combinatorial refraction billiards for other Coxeter groups besides the affine symmetric group. This definition relies on a fixed choice of a periodic infinite word in the simple reflections of the Coxeter group. It could be interesting to define and explore the analogues of ensnaring graphs and expelling graphs for other affine Weyl groups. It seems especially nice to consider the affine Weyl group $\widetilde C_n$, whose Coxeter graph is a path with vertices $s_0^C,s_1^C,\ldots,s_n^C$ (listed in order along the path). A natural periodic sequence of simple reflections that one could use to define refraction billiards is obtained by repeating the sequence \[s_0^C\,s_1^C,\ldots,s_{n-1}^C,s_n^C,s_{n-1}^C,\ldots,s_1^C.\] 

\section*{Acknowledgments}
Colin Defant was supported by the National Science Foundation under Award No.\ 2201907 and by a Benjamin Peirce Fellowship at Harvard University. He thanks Pavel Galashin for suggesting the topological perspective on toric combinatorial billiards that we have pursued. This research was conducted at the University of Minnesota Duluth REU with support from Jane
Street Capital, NSF Grant 1949884, and donations from Ray Sidney and Eric Wepsic. We thank Joe Gallian for providing this wonderful opportunity. We also thank Noah Kravitz, Mitchell Lee, and Maya Sankar for helpful discussions.

\end{document}